\documentclass[ ]{elsarticle}

\newtheorem{theorem}{Theorem}
\newtheorem{lemma}[theorem]{Lemma}

\newtheorem{remark}{Remark}
\usepackage{curves,amsmath,amssymb,color,fancybox,exscale,lineno,hyperref}
\usepackage{bm}
\usepackage{mathabx}
\usepackage{subfig}

\definecolor{dark}{gray}{0.6}
\definecolor{light}{gray}{0.8}

\def\norm#1{|\!| #1 |\!|}

\def\enorm#1{|\!|\!| #1 |\!|\!|}

\def\b1{{\mathbf 1}}
\def\ba{{\mathbf a}}
\def\bv{{\mathbf v}}
\def\bu{{\mathbf u}}

\def\bb{{\mathbf b}}
\def\bc{{\mathbf c}}

\def\br{{\mathbf r}}
\def\bs{{\mathbf s}}
\def\bd{{\mathbf d}}
\def\be{{\mathbf e}}
\def\bp{{\mathbf p}}
\def\bq{{\mathbf q}}
\def\bx{{\mathbf x}}
\def\by{{\mathbf y}}

\def\bn{{\mathbf n}}

\def\bS{{\mathbf S}}
\def\bH{{\mathbf H}}

\def\btau{{\boldsymbol \tau}}

\renewcommand{\P}{{\mathcal P}}

\def\Nedelec{N\'ed\'elec\ }

\def\rebAuthor{Randolph E.~Bank}

\def\psvAuthor{Panayot S.~Vassilevski}

\def\ltzAuthor{Ludmil T.~Zikatanov}

\def\rebAddress{Department of Mathematics, University of California, San Diego,
  La Jolla, CA~92093.}

\def\psvAddress{Center for Applied Scientific Computing, Lawrence Livermore
  National Laboratory,  P.O. Box 808, Mail Stop L-561, Livermore,
  CA~94551.} 

\def\ltzAddressa{Department of Mathematics, The Pennsylvania State University,  
  University Park, PA~16802.}
\def\ltzAddressb{Institute of Mathematics and Informatics, Bulgarian
  Academy of Sciences, 1113 Sofia, Bulgaria.}

\def\rebEmail{rbank@ucsd.edu}
\def\psvEmail{panayot@llnl.gov}
\def\ltzEmail{ludmil@psu.edu}

\def\rebThanks{National Science Foundation under  contract DMS-1318480.}

\def\ltzThanks{National Science Foundation under  contracts DMS-1418843 and DMS-1522615.}

\def\myTitle{Arbitrary Dimension Convection-Diffusion Schemes for 
  Space-Time Discretizations}

\def\myKeywords{space-time formulation, convection-diffusion problems, finite-element method, exponential fitting, streamline-diffusion}

\def\myAbstract{ 
  This note proposes embedding a time dependent PDE
  into a convection-diffusion type PDE (in one space dimension higher)
  with singularity, for which two discretization schemes, the
  classical streamline-diffusion and the EAFE (edge average finite
  element) one, are investigated in terms of stability and error
  analysis. The EAFE scheme, in particular, is extended to be
  arbitrary order which is of interest on its own. Numerical results,
  in combined space-time domain demonstrate the feasibility of the
  proposed approach.  
}

\begin{document}

\makeatletter
\def\@author#1{\g@addto@macro\elsauthors{\normalsize%
    \def\baselinestretch{1}%
    \upshape\authorsep#1\unskip\textsuperscript{%
      \ifx\@fnmark\@empty\else\unskip\sep\@fnmark\let\sep=,\fi
      \ifx\@corref\@empty\else\unskip\sep\@corref\let\sep=,\fi
    }%
    \def\authorsep{\unskip,\space}%
    \global\let\@fnmark\@empty
    \global\let\@corref\@empty  
    \global\let\sep\@empty}%
  \@eadauthor={#1}
}
\makeatother

\begin{frontmatter}
  \title{\myTitle\tnoteref{ztitlenote}} \tnotetext[ztitlenote]{This
    work was performed under the auspices of the U.S. Department of
    Energy by Lawrence Livermore National Laboratory under Contract
    DE-AC52-07NA27344.  The work of the first author was supported
    in part by the \rebThanks and the work of the third author was
    supported in part by the \ltzThanks}

  \author{\rebAuthor}
  \address{\rebAddress\, {Email:}\rebEmail}

  \author{\psvAuthor\corref{cor1}}
  \address{\psvAddress\, {Email:}\psvEmail}   

  \author{\ltzAuthor}
  \address{\ltzAddressa\, {Email:}\ltzEmail\\ \ltzAddressb}

  \cortext[cor1]{Corresponding Author}
  \begin{abstract}\myAbstract\end{abstract}

\end{frontmatter}



\section{Introduction}

The embedding of time-dependent problems into a one space dimension
higher stationary problem is not a new idea. It has many appealing
properties, such as: using already existing tools developed for
stationary problems; using adaptive methods with reliable and
efficient error control; the ability to use existing efficient solver
libraries developed for stationary problems. There is, however, a
drawback: typically, the memory needed to run a simulation using the
combined space-time discretization approach is increased by an order
of magnitude. One way to keep the memory required by such methods
under control is to use time intervals with fixed length.  Another,
more general, remedy to the extensive use of computer memory in
space-time simulations is to employ accurate dimension reduction
algorithms, both in space and in time, which can lead to coarser
problems with fewer degrees of freedom, also known as upscaled
discretizations. Indeed, an accurate coarser problem can replace the expensive, 
in terms of memory, fine-grid one and still provide a reliable
discretization tool.  For a general dimension reduction approach by
coarsening (in three space dimensions), we refer to
\cite{LashukVassilevski2014}.  The extension of the technique proposed
in~\cite{LashukVassilevski2014} to 4D space-time elements is a work
in progress.  Another feasible approach for dimension reduction in
space-time discretizations is to exploit sparse grids, as proposed
in~\cite{GOV_sparse_grid_space-time}.  More recently, discrete space-time
schemes using B-splines and Non-Uniform Rational Basis Splines (NURBS)
have been employed (see~\cite{UL:LangerMooreNeumueller:2015a}) to
yield stable isogeometric analysis methods for the numerical solution
of parabolic PDEs in fixed and moving spatial
domains. 

We point out that in the present note we do not consider dimension
reduction techniques. Rather, as a first step, we study the accuracy
and stability of the proposed embedding. More specifically, for the
discretization of the space-time formulation of a parabolic problem we
exploit two well-known techniques for convection diffusion equations:
the streamline diffusion method~\cite{1982BrooksA_HughesT-ab} (see
also \cite{1987JohnsonC-ab},
\cite{1989AxelssonO_EijkhoutV_PolmanB_VassilevskiP-aa}) and the
EAFE--Edge Average Finite Element scheme~\cite{1999XuJ_ZikatanovL-aa}
(see also~\cite{1998BankR_CoughranW_CowsarL-aa}
and~\cite{2005LazarovR_ZikatanovL-aa}).  Let us add that the high
order EAFE method developed here provides a novel, high order,
exponentially fitted discretization for convection-diffusion
problems with suitable stability and approximation properties. 

The structure of the remainder of this note is as follows. In
Section~\ref{section:intro}, we introduce the space-time formulation
of parabolic problems.  Then, in Section~\ref{section:SD}, we present
the streamline diffusion method in our space-time setting.
Section~\ref{section-div}, contains the derivation of the high order
EAFE scheme on simplicial finite element grids in arbitrary spatial
dimension. The application details for the lowest order EAFE
discretization to parabolic problems is given in
Section~\ref{section-app-parabolic}.  Finally, in
Section~\ref{section:numerical-tests}, we present numerical tests
showing the optimality and efficiency of both schemes for space-time
formulation of parabolic problems. We conclude this paragraph with
remark on the terminology: as the EAFE scheme may be viewed as a
multidimensional Scharfetter-Gummel
discretization~\cite{1969ScharfetterD_GummelH-aa}, in what follows, we
use the terms ``EAFE discretization'' and ``Scharfetter-Gummel
discretization'' interchangeably.

\section{Space time formulation of parabolic problems}\label{section:intro}

We consider the following parabolic problem:
\begin{equation}\label{eq-model}
\begin{array}{l}
 u_t - \operatorname{div}(K(x)\nabla u - \bm{\beta}\cdot u) + \gamma u= f,\quad x \in \Omega_s,\\
 u=0, \quad x\in \Gamma=\partial\Omega_s;\quad u(x,0) = u_0(x), \quad x\in \Omega_s.
\end{array}
\end{equation}
Here, $\bm{\beta}$ is a vector field (a velocity) and $K(x)$ is, in
general, a scalar (or $d\times d$ tensor valued) function.  Let
$\Omega_t=(0,t_{\max{}})$ be the time interval of interest. The
space-time domain is $\Omega=\Omega_s\times\Omega_t$.  For convenience
we have assumed homogeneous Dirichlet boundary conditions $u=0$ on
$\partial\Omega_s\times \Omega_t$.  In treating time as a space-like
variable, 
the
initial condition at $t=0$ becomes a Dirichlet boundary
condition for the $(d+1)$ dimensional problem.  

In a space-time formulation, introducing a new variable $y=(x,t)$ then
gives the following convection diffusion equation: Find $u=u(y)$ such that
\begin{eqnarray*}
&& -\operatorname{div}_y(D\nabla_y u + \bb\cdot\nabla_y u) +\gamma u
  = f \mbox{ in } ~\Omega
\quad \bb=(\bm{\beta}^t,1)^t: \Omega \mapsto \mathbb{R}^{d+1},\\
&& u = 0\mbox{ on }~\Gamma=\partial\Omega\times\Omega_t;\quad
 u = u_0\mbox{ on }~\Gamma_0=\Omega_s\times\{t=0\}.
\end{eqnarray*}
Without loss of generality we may assume that $u_0=0$ and we define
${\cal H}^1_E(\Omega)$ as the subspace of ${\cal H}^1(\Omega)$
satisfying these homogeneous Dirichlet boundary conditions.

In the following we consider two schemes for discretization of
convection diffusion problems and apply them to space-time
formulations of~\eqref{eq-model}. These are the Streamline Diffusion and
the Scharfetter-Gummel (EAFE) discretizations.  For the latter we need a
non-singular $D$, while above $D = \left [
\begin{array}{cc}
K & 0 \\
0 & 0
\end{array} \right ]$ is actually degenerate. To remedy this, we 
perturb it to make it invertible, i.e., we  let $D = \left [
\begin{array}{cc}
K & 0 \\
0 & \epsilon
\end{array} \right ]$ for a small parameter $\varepsilon>0$. 
\section{Streamline Diffusion}\label{section:SD}
We first consider a simple case when $K=\alpha I$, $\alpha>0$,
$\gamma\ge 0$, and $\bm{\beta}$ are constant. Then 
equation~\eqref{eq-model} has the form:
\begin{equation}\label{strongform}
Lu\equiv u_t  -\alpha\Delta u +  \bm{\beta}\cdot\nabla u +\gamma u =f
\end{equation}
The results below generalize to the variable coefficient case in a
straightforward and well-studied fashion. Here we consider the
constant coefficients case only in an attempt to keep the focus on the important
aspect of time discretization. In allowing for different sizes of
$\alpha$, $\bm{\beta}$ and $\gamma$, our analysis covers several
scenarios of interest. For simplicity we assume the initial condition
$u_0=0$.

The weak form of \eqref{strongform}  is given by: find $u\in{\cal H}^1_E$
such that
\[
(u_t,v)+\alpha(\nabla u,\nabla v)+(\bm{\beta}\cdot\nabla u,v)+ \gamma(u,v)=f(v)
\]
for all $v\in{\cal H}^1_E$. The space-time bilinear for $B(u,v)$ is given by
\[
B(u,v)=\int_0^T
(u_t,v)+\alpha(\nabla u,\nabla v)+(\bm{\beta}\cdot\nabla u,v)+ \gamma(u,v)\,dt
\]
where
\[
(u,v)=\int_{\Omega_s} uv\,dx
\]
is the usual ${\cal L}_2$ inner product on $\Omega_s$. 
The right hand side is given by the linear functional
\[
F(v)=\int_0^T f(v)\,dt.
\]

We assume that the space-time domain
$\Omega$ is covered by a shape regular quasiuniform tessellation ${\cal T}_h$
of elements of size $h$.
The energy norm for this problem is given by
\[
\enorm{u}^2=\norm{u(T)}^2+\int_0^T \alpha\norm{\nabla u}^2
+ h^p  \nu \norm{\bm{\beta}\cdot\nabla u +u_t}^2+\gamma\norm{u}^2\, dt
\]
where
\[
\nu=\frac{1}{\sqrt{|\bm{\beta}|^2+1}}.
\]
For technical reasons made clear below, we set $p=1$ for the important case
of continuous piecewise linear approximation, or the special case 
$\alpha=0$; otherwise we choose $p=2$.

We make a standard Petrov-Galerkin streamline diffusion discretization
for this $d+1$ dimensional problem. 
Let $V_h\subset{\cal H}^1_E$ 
denote a $C^0$ conforming piecewise polynomial
finite element space. The space $V_h$ itself is
the trial space.
In our Petrov-Galerkin  formulation, the test functions are given by
$v+\theta h^p\nu (\bm{\beta}\cdot\nabla v+v_t)$ for $v\in V_h$, 
where $\theta$ is a parameter to be  characterized below.
The discrete problem is: find $u_h\in V_h$ such that
\[
B_h(u_h,v)\equiv B(u_h,v)+\int_0^T(Lu_h,\theta h^p\nu (\bm{\beta}\cdot\nabla v+v_t))\,dt
        =F(v+\theta h^p\nu (\bm{\beta}\cdot\nabla v+v_t))
\]
for all $v\in V_h$. Because $V_h$ is only $C^0$, the term
$(Lu_h,\theta h\nu (\bm{\beta}\cdot\nabla v+v_t))$ is formally interpreted 
elementwise due to possible discontinuities on inter-element boundaries.

We begin with a basic stability result.
\begin{lemma}\label{stabilityresult}
Let $V_h$ be the space of continuous piecewise linear polynomials
or $\alpha=0$.
For $v\in V_h$, and $\theta$ sufficiently small, there exists
$C>0$, independent of $h$, such that
\begin{equation}\label{eqn11}
B_h(v,v) \ge C\enorm{v}^2.
\end{equation}
\end{lemma}
\begin{proof}
We first note the term 
$(\alpha\Delta v,\theta h\nu (v_t+\bm{\beta}\cdot\nabla v)=0$. 

The term
\[
\int_0^T(v_t+\bm{\beta}\cdot\nabla v, v)\,dt  =\frac{\norm{v(T)}^2}{2}
\]
and
\[
(v_t+\bm{\beta}\cdot\nabla v, \theta h\nu (v_t+\bm{\beta}\cdot\nabla v) =
\theta h\nu \norm{v_t+\bm{\beta}\cdot\nabla v}^2.
\]
Finally
\begin{align*}
(\gamma v,v+\theta h\nu (v_t+\bm{\beta}\cdot\nabla v))
&\ge \gamma\norm{v}^2-\gamma\theta h\nu\norm{v}\norm{v_t+\bm{\beta}\cdot\nabla v}
\\
&\ge \gamma\norm{v}^2\left(1-\frac{\gamma\theta h\nu}{2}\right)
-\frac{\theta h\nu}{2}\norm{v_t+\bm{\beta}\cdot\nabla v}^2
\end{align*}
Combining all these estimates, and taking $\theta$ sufficiently small
proves \eqref{eqn11}.
\end{proof}

The orthogonality-like relation for the error $e=u-u_h$ in our approximation
is given by
\begin{equation}\label{eqn12}
B_h(e,v)=0
\end{equation}
for all $v\in V_h$. 

For $\chi\in V_h$,
let 
\begin{gather*}
\phi=u_h-\chi \\ 
\eta=u-\chi. 
\end{gather*}
Our error relation can be expressed in terms of $\phi$ and $\eta$ as
\[
B_h(\phi,v)= B_h(\eta,v)  
\] 
for all $v\in V_h$. We take $v=\phi\in V_h$ and use
Lemma \ref{stabilityresult}. Then we have
\begin{equation}\label{eqn13}
\enorm{\phi}^2\le CB_h(\phi,\phi)\le CB_h(\eta,\phi)
\end{equation}
Let $\delta$ be a sufficiently small parameter to be characterized below.
We now estimate all the terms on the right hand side of
\eqref{eqn13}. First,
\begin{gather*}    
(\eta_t+\bm{\beta}\cdot\nabla \eta,\theta h\nu
 (\phi_t+\bm{\beta}\cdot\nabla \phi)) \le
C\theta h\nu\norm{\eta_t+\bm{\beta}\cdot\nabla \eta}^2
+\delta\theta h\nu\norm{\phi_t+\bm{\beta}\cdot\nabla \phi}^2 
\\
(\gamma \eta,\phi+\theta h\nu(\phi_t+\bm{\beta}\cdot\nabla \phi))
\le C\gamma \norm{\eta}^2+ \delta(\gamma\norm{\phi}^2+\theta h\nu
\norm{\phi_t+\bm{\beta}\cdot\nabla \phi}^2)
\end{gather*}
The third term is a bit more involved.
\begin{align*}
\int_0^T(\eta_t+\bm{\beta}\cdot\nabla \eta,\phi)\,dt &=
(\eta(T),\phi(T))-\int_0^T(\eta,\phi_t+\bm{\beta}\cdot\nabla\phi)\,dt\\
  &\le C\left(\norm{\eta(T)}^2+(h\nu)^{-1}\int_0^T\norm{\eta}^2\,dt \right)\\
      & \hspace{0.5cm}+\delta\left(\norm{\phi(T)}^2+h\nu\int_0^T
                \norm{\phi_t+\bm{\beta}\cdot\nabla\phi}^2\,dt\right)
\end{align*}
Combining these estimates, 
and making $\delta$ sufficiently small, we have
\begin{equation}\label{eqn14}
\enorm{\phi}^2\le C\left(\enorm{\eta}^2
+\int_0^T (h\nu)^{-1} \norm{\eta}^2\,dt\right)
\end{equation}
Using \eqref{eqn14} and the triangle inequality, we obtain
\begin{theorem}\label{streamlinetheorem}
Let $V_h$ be the space of continuous piecewise linear polynomials
or $\alpha=0$.
Then the error $e=u-u_h$ satisfies
\begin{equation}\label{eqn15}
\enorm{u-u_h}^2\le C \inf_{\chi\in V_h}
\left(\enorm{u-\chi}^2+\int_0^T (h\nu)^{-1}
\norm{u-\chi}^2
\,dt  \right)
\end{equation}
\end{theorem}
Suppose $\alpha=O(1)$ and $V_h$ contains piecewise linear  polynomials.
 Then if $u\in {\cal H}^2(\Omega)$, \eqref{eqn15}
yields an $O(h^{1/2})$ rate of convergence of the space-time
gradient in the streamline direction $\bb=(\bm{\beta}^t ,1)^t$, and 
an optimal $O(h)$ convergence rate for  
$(\int_0^T\norm{\nabla(u-u_h)}^2)^{1/2}$.
If $\gamma=O(1)$ we have $O(h)$ convergence for 
the space-time ${\cal L}_2$ norm.
While not optimal in every norm considered,
overall this is in alignment with well-known
behavior for the classical streamline diffusion method. If
$\alpha=0$ and $V_h$ contains piecewise polynomials of degree $r$,
we lose control of the  gradient $\norm{\nabla(u-u_h)}$
 but gain improved $O(h^r)$ convergence for the space-time gradient in the 
streamline direction, and
if $\gamma=O(1)$ we have improved   $O(h^{r+1/2})$ convergence 
for the ${\cal L}_2$ norm. These again correspond with classical results 
for the streamline diffusion method.

If $\alpha\ne 0$ and $V_h$ contains piecewise polynomials of degree $r>1$,
terms similar to $\alpha(\Delta v,v_t)$ become problematic 
since $\Delta v$ contains no time derivatives and at present
forces us to choose $p=2$. 
(One might alternatively consider replacing the diffusion term
$\Delta u$ with $\Delta u+ \epsilon u_{tt}$, and then analyzing as in the
standard streamline diffusion scenario, but this dilutes the
advantage one obtains through the use of higher order approximation since
we lose consistency with the original PDE).
Here is the analog of Lemma \ref{stabilityresult}.
\begin{lemma}\label{stabilityresult0}
Let $V_h$ be the space of continuous piecewise linear polynomials of degree
$r>1$ and $\alpha\ne 0$.
For $v\in V_h$, and $\theta$ sufficiently small, there exists
$C>0$, independent of $h$, such that
\begin{equation}
B_h(v,v) \ge C\enorm{v}^2.
\end{equation}
\end{lemma}
\begin{proof} Generally the proof follows the same pattern as
Lemma \ref{stabilityresult}. The new term is
$(\Delta v,\theta h^2\nu (v_t+\bm{\beta}\cdot\nabla v)$. 
On a single element $\tau\in{\cal T}_h$
we can use a local inverse assumption 
\[
|(-\Delta v,v_t+\bm{\beta}\cdot\nabla v)_{\tau}
        \le Ch^{-1}\norm{\nabla v}_{\tau}\norm{v_t+\bm{\beta}\cdot\nabla v}_{\tau}|
\]
Using this estimate, we have
\begin{multline*}
\int_0^T \alpha (\nabla v,\nabla v)
-(\alpha\Delta v,\theta h^2\nu (v_t+\bm{\beta}\cdot\nabla v))\,dt \\
\ge \int_0^T   \alpha (1-C\alpha\nu \theta) \norm{\nabla v}^2\,dt 
-\frac{\theta h^2\nu}{4}\norm{v_t+\bm{\beta}\cdot\nabla v}^2.
\end{multline*}
The remaining estimates in the proof of Lemma \ref{stabilityresult} are the
same with $h$ replaced by $h^2$.
\end{proof}

We analyze the error similar to proof of Theorem \ref{streamlinetheorem}.
The new terms are
\[
\alpha (\nabla \eta,\nabla\phi) \le
C\alpha \norm{\nabla \eta}^2+\delta\alpha \norm{\nabla \phi}^2
\]
and
\[
\alpha (\Delta \eta,\theta h^2\nu (\phi_t+\bm{\beta}\cdot\nabla \phi))\le
C\alpha^2\nu h^2\norm{\Delta \eta}^2+\delta h^2\nu
\norm{\phi_t+\bm{\beta}\cdot\nabla \phi}^2.
\]
The remaining terms are estimated as in
Theorem \ref{streamlinetheorem} with $h$ replaced by $h^2$, leading to
\begin{theorem} 
Let $V_h$ be the space of continuous piecewise linear polynomials of degree
$r>1$ and $\alpha\ne 0$.
Then the error $e=u-u_h$ satisfies
\begin{equation}\label{eqn15a}
\enorm{u-u_h}^2\le C \inf_{\chi\in V_h}
\left(\enorm{u-\chi}^2+\int_0^T (h^2\nu)^{-1}
\norm{u-\chi}^2
+\alpha h^2\norm{\Delta(u-\chi)}^2\,dt  \right)
\end{equation}
\end{theorem}

If $\alpha=O(1)$, $V_h$ is the space of 
continuous polynomials of degree $r>1$, and
$u$ is sufficiently smooth, we have optimal $O(h^r)$ convergence for
$(\int_0^T \norm{\nabla (u-u_h)}^2)^{1/2}$, but only $O(h^{r-1})$ 
convergence for  the space-time gradient in the streamline direction.
If $\gamma=O(1)$ we also obtain $O(h^r)$ convergence in the space-time
${\cal L}_2$ norm.

\section*{A Practical Remark}

Suppose that the space domain $\Omega_s$ has a generic length scale 
$L$. Since the time units for $\Omega_t=[0,T]$ could be completely
unrelated to the space units, the space-time domain 
$\Omega=\Omega_s\times\Omega_t$ could be
quite anisotropic. It could be very long if $T\gg L$ or very short if $T\ll L$.
Filling such potentially thin domains with a small 
number of shape regular elements
could be problematic from the practical point of view. Therefore it could
be useful to rescale the time variable such that it is has a similar scale to
the space variables. For example, one could change variables as in
\[
\tilde{t}=\frac{Lt}{T}\equiv \kappa t
\]
for $0\le\tilde{t}\le L$.
The modified space time-domain $\Omega_s\times [0,L]$ is more
isotropic, and likely could be tessellated with far fewer shape regular
elements.
In terms of the partial differential equation,
\[
\frac{\partial u}{\partial t}=
\kappa   \frac{\partial u}{\partial \tilde{t}}
\]
making the convection in the time direction larger or smaller depending
on the value of $\kappa$. In terms of our analysis, we could replace
\begin{gather*}
\alpha \rightarrow \frac{\alpha}{\kappa}\equiv\tilde{\alpha} \\
\beta  \rightarrow \frac{\beta }{\kappa}\equiv\tilde{\beta} \\
\gamma \rightarrow \frac{\gamma}{\kappa}\equiv\tilde{\gamma}
\end{gather*}
and directly apply the analysis of the previous section to this
modified constant coefficient equation.
\section{High Order Scharfetter Gummel discretization}\label{section-div}
In this section we derive a high order Scharfetter-Gummel scheme on
simplicial finite element grids in dimension $d\ge 1$. The original
Scharfetter-Gummel difference
scheme~\cite{1969ScharfetterD_GummelH-aa} is a method used in
simulating 1-dimensional semiconductor equations. After its discovery,
it has been generalized and used for the numerical solution of
convection-diffusion equations of the form:
\begin{eqnarray}
&&\label{div-form} -\operatorname{div} J(u) = f, \quad x\in
   \Omega\subset \mathbb{R}^d \\ 
&&\label{flux} J(u) = (D(x)\nabla_x  u - \bb u), \\
&&\label{bca} u(x)=0,\quad x\in \Gamma_D, \quad J(u)\cdot\bn =0,\quad x\in \Gamma_N\\
&& \label{bcb} Du\cdot\bn =0,\quad x\in \Gamma_R.
\end{eqnarray}
Here, $J(u)$ is the flux variable which plays an important role in
approximating the weak form of the equation. We note that the natural
boundary condition is the one given on $\Gamma_N$ and the boundary
condition~\eqref{bcb} is of a Robin type for this problem. 
The weak form of the equation above is: Find $u\in V$ such that 
\begin{eqnarray}
&&\label{weak} a(u,v) + m_{R}(u,v)= f(v), \\
&&\label{weak-div-form} a(u,v) = \int_\Omega J(u)\cdot\nabla v, \quad f(v)
   = \int_{\Omega} fv\\
&&\label{weak-bcb} m_{R}(u,v) = \int_{\Gamma_R}(\bb \cdot\bn )uv
\end{eqnarray}
The variational form is obtained after integration by parts and using
the fact that on $\Gamma_R$,
$Du\cdot\bn =J(u)\cdot\bn  - \bb \cdot\bn  u$.

The Scharfetter-Gummel scheme was extended to more than 1 spatial
dimension as the Edge Average Finite Element (EAFE) Scheme.  A priori
error estimates in any dimension were shown
in~\cite{1999XuJ_ZikatanovL-aa}. This work only considered scalar
valued diffusion coefficients (although in any spatial dimension); a
discretization for matrix valued diffusion coefficients was proposed
and analyzed in~\cite{2005LazarovR_ZikatanovL-aa}.  Related work on
exponential fitting in discretizing convection-diffusion equations via
mixed finite element methods is~\cite{1989BrezziF_MariniL_PietraP-aa}.
More recently, the techniques from~\cite{1999XuJ_ZikatanovL-aa} have
been utilized to yield a second order gauge invariant discretizations
for Pauli and Schr\"{o}dinger equations
(see~\cite{2015Harald-ChristiansenS_HalvorsenT-aa}).

Here, we provide a novel approach which gives a Scharfetter-Gummel
discretization for finite element spaces of order $r\ge 1$.  Our
approach follows the ideas in~\cite{1999XuJ_ZikatanovL-aa} and
\cite{2005LazarovR_ZikatanovL-aa}. The extension to $r\ge 1$, however
is not at all straightforward and requires results from the recently
developed Finite Element Exterior Calculus. The rationale of
constructing the high order Scharfetter-Gummel scheme is:
\begin{itemize}\setlength{\itemsep}{0mm}
\item [(i)] approximate the flux $J(u)$ via the \Nedelec elements
  (discrete differential $1$-forms with polynomial coefficients);
\item [(ii)] eliminate the flux variable and write the resulting
  discrete problem in terms of the scalar valued finite element
  approximation of the solution of~\eqref{eq-model} $u$ (a $0$-form).
\end{itemize}
To set up the finite element approximation, we let us itemize some of the 
ingredients and the main assumptions needed for the discretization.
\begin{itemize}
\item We assume that $\Omega$ is covered by a conforming, simplicial,
  shape-regular mesh $\mathcal{T}_h$. We have 
$\Omega=\cup\{T\;\big|\;T\in\mathcal{T}_h\}$.
\item The space $V_h$ is the space of conforming Lagrange finite
  elements of degree $r$ and  for the
  derivation of the scheme, we also need the 1st-kind-\Nedelec polynomial spaces on a fixed
  element
  (cf.eg. \cite{1980NedelecJ-aa,1986NedelecJ-aa,1988BossavitA-aa,2003MonkP-aa}). The
  details are described below in~\S\ref{section:Nedelec}.
 
\item We assume that the flux $J$ and $u$ are smooth enough so that
  all the norms of functions below make sense. In particular
  $J\in W^{1,p}(T)$, for all $T\in \mathcal{T}_h$ and for some
  $p>d$. The solution $u$ is at least continuous, so that its Lagrange
  interpolant is well defined. 
\item We assume that the coefficients $D$, $\bb$ are piece-wise constants
  with discontinuities aligned with $\mathcal{T}_h$. 
\end{itemize}

\begin{remark}\label{rem1}
The assumption $J\in W^{1,p}(\Omega)$ needs some comments. 
One important feature of the Scharfetter--Gummel scheme is that the estimates
on $\|u_I-u_h\|_{1,\Omega}$ are in terms of norms of the
flux $J(u)$. We thus approximate more accurately the interpolant
$u_I\in V_h$ of the solution if the flux is smooth, while both the
solution and the coefficients can be rough functions. For example, if
we look at the 1D problem on $(0,1)$:
\[
-(u'-\beta u)' = 0, \quad u(0)=0, \quad u(1)=1,
\]
we observe that $J(u)=(u'-\beta u)$ is a constant, i.e. smooth, while
the solution $u$ may exhibit a sharp boundary layer, depending on
$\beta$. In fact, in this idealized situation in 1D, the estimate in 
Theorem~\ref{thm:ez} implies that
$u_I = u_h$, i.e. we have captured the exact solution at the vertices.
\end{remark}

Next, we show that (i) and (ii) in the rationale given earlier 
are computationally feasible steps. 

\subsection{Notation and \Nedelec spaces}\label{section:Nedelec}
Consider the \Nedelec space $P^{\mathcal{N}}$, which restricted to any
element $T$ is the following polynomial space
\begin{equation}\label{nedelec-space}
\P^{\mathcal{N}} = \left (P_{r-1} \right )^d \oplus {\bS}_{r}, \quad  
P_{r-1}^d\subsetneq P^{\mathcal{N}} \subsetneq (P_{r})^d,
\end{equation}
where $P_j$, $j=(r-1), r$ is the space of polynomials of degree $\le j$ on $T$, and 
${\bS}_{r}$ is a subspace of the space $\bH_r$ of vector valued homogeneous polynomials of
degree $r$ defined as
\begin{equation*}
\bS_r = \left\{\bs\in \bH_r \;\;\big|\;\; \bs\cdot \bx = 0\right\}.
\end{equation*}
By definition, the inclusion relations given in
equation~\eqref{nedelec-space} hold on any element
$T\in \mathcal{T}_h$.  From now on we fix this element. We refer
to~\cite{1980NedelecJ-aa}, \cite{1986NedelecJ-aa},
\cite{1999HiptmairR-aa}, \cite{2006ArnoldD_FalkR_WintherR-aa},
\cite{2010ArnoldD_FalkR_WintherR-aa}, for the classical and the modern
description of these spaces and studies of their properties. In what
follows we use some of the tools from~\cite{1999HiptmairR-aa} and
\cite{2006ArnoldD_FalkR_WintherR-aa}.  In our notation, the lowest
order of such polynomials corresponds to $r=1$.

Further, let $M=\dim P^{\mathcal{N}}$ be the dimension of the \Nedelec
polynomial space on $T$.  The elements of the basis in the dual space
of $P^{\mathcal{N}}$ are known as \emph{degrees of freedom} and we
denote them by $\{\eta_j\}_{j=1}^M$. Next, the basis in
$P^{\mathcal{N}}$, dual to the degrees of freedom we denote by
$\{\bm{\varphi}_j\}_{j=1}^M$. For general simplex in $\mathbb{R}^d$,
the explicit form of the degrees of freedom and their dual basis is
found in~\cite{1999HiptmairR-aa}. For our purposes it is sufficient to
note that the functionals $\eta_j$ can be thought as integrals of
traces of functions over sub-simplices.  For the lowest order case,
we have
\[
\langle \eta_e,\bv\rangle = \int_{e} \bv \cdot \bm{\tau}_e, \quad 
\bm{\varphi}_e = \lambda_i\nabla\lambda_j-\lambda_j\nabla \lambda_i. 
\]
for every edge $e=(i,j)$ of $T$ (there are $\frac{d(d+1)}{2}$ edges).
Here, $\bm{\tau}_e$ is the tangent for edge $e$, and $\{ \lambda_i\} $
are the usual barycentric coordinates for element $T$ (cf., e.g.,
\cite{2002CiarletP-aa}).
Using this notation, we have that any function
$\bv\in P^{\mathcal{N}}$ can be written as
\begin{equation}\label{expansion}
\bv = \sum_{j=1}^M \langle \eta_j, \bv\rangle\bm{\varphi}_j(\bx ). 
\end{equation}
We stress that this representation is unique and provides a canonical interpolation operator,
which for sufficiently smooth vector valued $\bv$ is defined as 
\begin{equation}\label{interpolation-n}
\Pi^{\mathcal{N}}\bv = \sum_{j=1}^M \langle \eta_j,
\bv\rangle\bm{\varphi}_j(\bx ). 
\end{equation}
The smoothness of $\bv$ must be such that the linear forms 
$\langle \eta_j,\cdot\rangle$ are bounded.  

Consider now the space $V_h$ of Lagrange finite elements of order $r$.
The standard set of the degrees of freedom in such case are
point evaluations (see \cite[Theorem~2.2.1]{2002CiarletP-aa}) and we denote them by $\{\mu_j\}$. Further, the polynomial basis, dual to these degrees of freedom, we denote by
$\{\xi_j\}$. We then have a canonical interpolation operator, well defined
for any continuous $v$. The image of $v\in C^{0}(\overline\Omega)$
under this interpolation is denoted by $v_I$ and we have 
\begin{equation}\label{interpolation-l}
v_I = \sum_{j=1}^{N_h} \langle \mu_j, v \rangle \xi_j (\bx ). 
\end{equation}
There is no need to distinguish the global interpolation operator (on
$\Omega$) and the local one (on $T\in \mathcal{T}_h$) for our
considerations and we use the same notation for both. Let us note,
however, that when working on fixed $T\in \mathcal{T}_h$ we will use
$N=\dim P_r={r+d \choose d}$, instead of $N_h=\dim V_h$.
 
As is well known~\cite{1999HiptmairR-aa}, we have commutative diagrams
linking the \Nedelec elements and the Lagrange elements of matching
orders (order $r$ here), and on every element $T$ we have
\[
\Pi^{\mathcal{N}} \nabla v =  \nabla v_I. 
\]
This relation is in fact a relation between degrees of freedom, namely
\begin{equation}\label{crucial-a}
\langle \eta_j , \nabla v \rangle =  \langle\eta_j, \nabla v_I\rangle.
\end{equation}
This is obvious by using the definition of 
$\Pi^{\mathcal{N}}$,  the fact that $\nabla v_I\in P^{\mathcal{N}}$,
and the uniqueness of the representation in~\eqref{expansion}.


\subsection{Derivation of a high order Scharfetter-Gummel scheme}\label{derivation-section}
Let us fix $T\in \mathcal{T}_h$ and 
we start with the definition of $J$ and use that $D$ and $\bb$ are
constant on $T$. 
\begin{eqnarray*}
&&J(u) = D\nabla u - \bb u = \exp(\bq\cdot \bx)D\nabla (\exp(-\bq\cdot
   \bx) u),\;\; \bq = D^{-1} \bb.
\end{eqnarray*}
Hence, we have 
\begin{equation}\label{grads}
\exp(-\bq\cdot \bx)D^{-1} J(u) = \nabla (\exp(-\bq\cdot \bx) u). 
\end{equation}
If we apply now $\langle \eta_j,\cdot\rangle$, $j=1:M$ on both sides,
and then use~\eqref{crucial-a} we get
\begin{eqnarray*}
\langle\eta_j, e^{(-\bq\cdot \bx)}D^{-1} J(u)\rangle & = & 
\langle\eta_j, \nabla (e^{(-\bq\cdot \bx)} u) \rangle \\
& = & \langle\eta_j, \nabla \big(e^{(-\bq\cdot \bx)} u\big)_I\rangle ,\\
& = & \langle\eta_j, \nabla \big(e^{(-\bq\cdot \bx)} u_I\big)_I\rangle ,
\end{eqnarray*}
The latter identity on the right hand side above, uses the fact that the $''I''$-interpolant is based on the functionals $\mu_j$ that are based on nodal evaluation.
As expected, the right hand side is a gradient of a function in $V_h$
and in summary we have
\begin{equation}\label{the-identity}
\langle\eta_j, e^{(-\bq\cdot \bx)}D^{-1} J(u)\rangle 
= \langle\eta_j, \nabla \big(e^{(-\bq\cdot \bx)} u_I\big)_I\rangle,\quad j=1,\ldots,M. 
\end{equation}
Introducing now $\bm{G}(J(u))\in \mathbb{R}^M$ and 
$\bd(u)\in \mathbb{R}^M$ by 
\begin{eqnarray}
\label{G-definition}
[\bm{G}(J(u))]_j & = & \langle\eta_j, e^{(-\bq\cdot \bx)}D^{-1} J(u)\rangle, \quad
j=1,\ldots,m\\
  \label{d-definition}
[\bd(u)]_j & = & \langle\eta_j, \nabla \big(e^{(-\bq\cdot \bx)}
u_I\big)_I\rangle,\quad j=1,\ldots,M.
\end{eqnarray} 
and we can write~\eqref{the-identity} as
\begin{equation}\label{the-identity-a}
\bm{G}(J(u)) = \bd(u),
\end{equation}
Note that both $\bm G$ and $\bd$ are linear operators, mapping vector
fields and functions to $\mathbb{R}^M$. We remark that the
relation~\eqref{the-identity-a} is used later in the definition of the
approximate bilinear form, and, in the proof of the error estimates,
and, we further stress on the fact that $\bd(u) = \bd(u_I)$, by definition.
The advantage of the exponential weighting in~\eqref{the-identity} is
that it provides a way to accurately approximate the flux by
polynomials. Note that the right side of~\eqref{the-identity} are the
values of $\eta_j$ evaluated on a polynomial, while the left side
contains the ``true'' flux $J(u)$.  As we shall see later, it is
advantageous to use discretizations schemes based
on~\eqref{the-identity} when the flux is piece-wise smooth. This not
only includes the case when $u$ and the PDE coefficients are
piece-wise smooth, but also includes many other cases (see
Remark~\ref{rem1} for a simple 1D example on this).

The main idea of the Scharfetter-Gummel and EAFE schemes is to 
approximate $J(u)$,
\[
J(u)\approx J_T(u)\in P^{\mathcal{N}},
\] 
or equivalently, we seek
\begin{equation*}
J_T(u) = \sum\limits c_j \bm{\varphi}_j,
\end{equation*}
for some coefficient vector $\bc = (c_j)$. The coefficient $\bc$ is chosen
so that the relation~\eqref{the-identity} still holds for the
approximation.  An important question is whether this is possible.  If
$r=1$, and we use the lowest order \Nedelec elements, this is
definitely the case as shown in the earlier
works~\cite{1999XuJ_ZikatanovL-aa, 2005LazarovR_ZikatanovL-aa}. 

A construction of exponentially-fitted discretizations with higher
order polynomial spaces is a bit more intricate.  In general, we would
like to find $J_T(u) \in P^{\mathcal{N}}$. A key observation is that
in order to derive our scheme, we use the weak form of the
equation~\eqref{div-form} and we will aim to approximate the weak form
as follows:
\begin{equation*}
\int\limits_T J(u_h) \cdot \nabla v_h\approx \int\limits_T J_T(u_h) \cdot \nabla v_h,
\end{equation*}
for functions $v_h \in V_h$ and $J_T(u_h)\approx J(u_h)$. As on $T$,
$\nabla v_h \in \left (P_{r-1} \right )^d$, it is sufficient to look
for approximations $J_T(u)\in (P_{r-1})^d\subset P^{\mathcal{N}}$.


We now explore this observation and look at how it
  affects the identity~\eqref{the-identity-a}.  Let $P$ be the matrix
  representation of the embedding
  $(P_{r-1})^d\subset P^{\mathcal{N}}$.  To define this matrix, let
  $\{\bm{\psi}_m\}_{m=1}^{M_0}$ be a basis in $(P_{r-1})^d$, with
  $M_0=\dim(P_{r-1})^d$ and let $\{\bm{\varphi}_j\}_{j=1}^{M}$ be the
  basis in $P^{\mathcal{N}}$, dual to the degrees of freedom
  $\{\eta_m\}_{m=1}^M$. Then the entries of $P$ are the coefficients
  in the expansion of $\bm\psi_m$ in terms of
  $\{\bm{\varphi}_j\}_{j=1}^M$, and we have,
\begin{equation}\label{prolongation}
\bm\psi_k = \sum_{j=1}^M p_{jk}\bm\varphi_j, \quad \mbox{with}\quad 
p_{mk} = \langle \eta_m,\bm\psi_k\rangle.
\end{equation}
Note that $\nabla \big(\exp(-\bq\cdot \bx) u_I\big)_I$ is an element
$(P_{r-1})^d$, and, as such, it can be written as a linear combination via
$\{\bm{\psi}_k\}_{k=1}^{M_0}$. Recalling the definition of $\bd(u)$
in~\eqref{d-definition} then leads to the following useful relations:
\begin{eqnarray*}
&&\nabla \big(\exp(-\bq\cdot \bx) u_I\big)_I =  \sum_{k=1}^{M_0}
  \widetilde{d}_k \bm\psi_k,\quad \lbrack\bd(u)\rbrack_j =   \sum_{k=1}^{M_0}\langle\eta_j,\bm\psi_k\rangle
                 \widetilde{d}_k,\\
&& \bd(u) = P \widetilde{\bd}.
\end{eqnarray*}
As a consequence, to define the approximation $J_T$, we need to find a
solution of the following problem
\begin{equation}\label{subspace solution}
P^*ZP {\widetilde \bc} = P^* P\widetilde{\bd},
\end{equation}
where we have set $\bc = P {\widetilde \bc}$,  and, as we have shown,  
$\bd=\widetilde{\bd}$. 
Above the matrix $Z\in \mathbb{R}^{M\times M}$ has entries 
\[
Z_{jk}=\langle \eta_j, e^{-\bb\cdot
  D^{-1}\bx}D^{-1}\bm{\varphi}_k\rangle.
\]

The following remark is in order.  In general, we may have tried to
solve the following system of equations for the coefficients $\bc$:
\begin{equation}\label{coeff}
Z\bc = \bd.
\end{equation}
Clearly, if this is a well posed problem, then we can find the
approximation $J_T(u)$. However, as it described above, we only need
the solution in subspace, i.e., to solve problem \eqref{subspace
  solution}.  
To show that the subspace problem~\eqref{subspace solution} is solvable it is
sufficient to show that $ZP$ is injective, and since $P$ is injective,
it is sufficient to show that $Z$ is injective on the range of $P$.

We let
\begin{equation}\label{Z dagger}
Z^{\dagger} = P \left (P^*ZP \right )^{-1} P^*.
\end{equation}
In the following, we will simply denote $\bc=P\widetilde{\bc}$ by
$\bc=Z^{\dagger}\bd$, or, by~\eqref{the-identity-a}, by $Z^{\dagger} G(J(u))$.

The following lemma follows by the construction of the approximation $J_T(u)$.
\begin{lemma}\label{lemma: polynomial recovery}
If $J(u)$ is polynomial of degree $r-1$, then its approximation 
$J_T(u)$ defined by $Z^{\dagger} G(J(u))$ coincides with $J(u)$.
\end{lemma}

\subsubsection{A unisolvence result}
Note that, when $\bb=0$, the solvability of such system follows from
the fact that the \Nedelec degrees of freedom form a unisolvent set of
functionals on $P^{\mathcal{N}}\supset (P_{r-1})^d$.  Multiplying by
the exponent changes the game, and, we need to prove some of the basic
results on unisolvence of \Nedelec degrees of freedom for
quasi-polynomials which we state in the following lemma.
\begin{lemma}\label{lemma-unisolvence} 
  The matrix $ZP$ is injective, or, equivalently, if
  $\bp\in (P_{r-1})^d$ and
  $\langle \eta_j,e^{(-\bb\cdot D^{-1}\bx)}\bp\rangle=0$, for all
  $j=1:M$, then $\bp=0$.
\end{lemma}
\begin{proof}
  This proof follows exactly the lines of the proofs of
  \cite[Lemma~4.5, Lemma~4.6]{2006ArnoldD_FalkR_WintherR-aa}. The only
  modifications involve the weighting with exponential
  function, which is positive everywhere. The rest of the arguments
  carry over without changes.  A proof in terms of vector proxies
  is given in~\ref{appendix}.
\end{proof}

\subsubsection{Derivation of the discrete problem}
Since now the approximation $J_T(u_h)$, for $u_h\in V_h$ is well
defined, due Lemma~\ref{lemma-unisolvence}, we have a natural
approximating bilinear form. For $u_h\in V_h$ and $v_h\in V_h$ we set
\renewcommand{\arraystretch}{1.5}
\begin{equation}\label{bilinear-approx}
\begin{array}{rcl}
\displaystyle a_h(u_h,v_h) & = & \displaystyle \sum_{T} \int_T J_T(u_h) \cdot\nabla v_h\\
& = & \displaystyle \sum_T \sum_{j=1}^M \int_T[Z_T^{\dagger} d_T(u_h)]_j\int_T
  \bm{\varphi}_j\cdot\nabla v_h. 
\end{array}
\end{equation}
\renewcommand{\arraystretch}{1}
 The coefficients in $J_T(u_h)$ are determined by $Z_T^{\dagger}\bd_T(u_h)$,
for all $T\in \mathcal{T}_h$, which in turn indeed makes the right
side of~\eqref{bilinear-approx} to depend only on the degrees of
freedom of $u_h$.

We then define the following discrete problem: Find $u_h\in V_h$ such that 
\begin{equation}\label{discrete-problem}
a_h(u_h,v) = f(v), \quad \mbox{for all }\quad v\in V_h.
\end{equation}

We note another useful relation which follows from the derivation
above and is used in the error estimates below. It is an analogue of
\cite[Equation~(3.16)]{1999XuJ_ZikatanovL-aa}, and
\cite[Equation~(3.8)]{2005LazarovR_ZikatanovL-aa} and it plays a
crucial role in the a priori error estimates.  In particular it is
useful to estimate the deviation of the derived discrete scheme from
the standard Galerkin one (with bilinear form
$a(u,v)=\int\limits_\Omega \left (D \nabla u - \bb u\right) \cdot
\nabla v$).

\begin{lemma}\label{identity-b} For any continuous $u$, and
  sufficiently smooth $J$, such that $\Pi^{\mathcal{N}}J(u)$ is well defined we have:
\begin{equation}\label{id-b}
a_h(u_I,v_h) = \sum_T \sum_{j=1}^M [Z_T^{\dagger} G_T(J(u))]_j\bm{\varphi}_j\cdot\nabla v_h. 
\end{equation}
\end{lemma}
\begin{proof}
Recalling that $d_T(u_I)=d_T(u)$, and
  substituting~\eqref{the-identity-a} in \eqref{bilinear-approx} gives the desired result. 
\end{proof}

Next, we show that, under certain conditions, this is a well posed
problem, and we also prove an \emph{a priori} error estimate.

\subsection{Stability and error analysis}
Error estimates and other properties of such discretization schemes
are found in~\cite{1999XuJ_ZikatanovL-aa},
\cite{2005LazarovR_ZikatanovL-aa}. Here we give an estimate for higher
order Scharfetter-Gummel discretization and assume for simplicity, and
without loss of any generality that we have Dirichlet boundary
conditions.  We have the following theorem:
\begin{theorem}\label{thm:ez}
 Assume that $a(\cdot,\cdot)$ is invertible on $V_h$. Then, 
for sufficiently small $h$, the discrete variational
  problem~\eqref{discrete-problem} is well posed and the following 
error estimate holds:
\begin{equation}\label{estimate-z}
|u_I - u_h|_{1,\Omega} \le ch^r |J(u)|_{r,p,\Omega}.
\end{equation}
\end{theorem}
\begin{proof}
From the definition of $a_h(\cdot,\cdot)$, for all $v\in V_h$ we have
\begin{eqnarray*}
|a(u,v) - a_h(u_I,v)| & = & 
\left|
\sum_{T} a_{T}(u,v)-a_{h,T}(u_I,v) 
\right|\\ 
& \le & 
\sum_{T}
\left| \int_{T} J(u)\cdot\nabla v - 
\sum_{j=1}^M [Z_T^{\dagger} G_T(J(u))]_j\int_T\bm{\varphi}_j\cdot\nabla v. 
\right|.
\end{eqnarray*}
Note that from Lemmas~\ref{identity-b} and \ref{lemma: polynomial recovery}, 
 the right side vanishes for all
$J(u)$ that are polynomials of degree $(r-1)$, and, standard scaling
argument shows the estimate
\begin{equation}\label{estimate-a}
|a(u,v) - a_h(u_I,v)| \le ch^r |J(u)|_{r,p,\Omega}|v|_{1,q,\Omega},
\quad p^{-1}+q^{-1} = 1. 
\end{equation}

The solvability of the discrete problem then follows from the fact
that by assumption $a(\cdot,\cdot)$ provides a solvable problem, and
hence it satisfies an inf-sup condition on $V_h$. According
to~\eqref{estimate-a} $a(\cdot,\cdot)$ and $a_h(\cdot,\cdot)$ are
close when $h\to 0$, and, hence, $a(\cdot,\cdot)$ also satisfies an
inf-sup condition for sufficiently small $h$. This in turn implies
that the discrete problem~\eqref{discrete-problem} is well posed. The
error estimate~\eqref{estimate-z} then follows from the
inequality~\eqref{estimate-a}. 
\end{proof}
\begin{remark}
  For the case $r=1$ our proof here is analogous to the one given
  in~\cite[Lemma~6.2 and Theorem~6.3]{1999XuJ_ZikatanovL-aa}.
\end{remark}

\section{Application to parabolic problems}\label{section-app-parabolic}
In this section, we recall the parabolic equation~\eqref{eq-model}:
\begin{equation}\label{parabolic-div-form}
\begin{array}{rcl}
&&u_t-\operatorname{div} (K(x)\nabla_x  u -\bm{\beta} u) = f, \\
&& u(x,0)=u_0(x),\quad \mbox{for}\quad t=0; \\
&&u(x,t)=0,\quad x\in \Gamma=\partial\Omega\times\{[ 0,t_{\max{}})\}.
\end{array}
\end{equation}
This equation and the
equation discretized by the streamline diffusion method match, if
$\operatorname{div}_x\bm{\beta} =0$, which we assume to hold.  In general,
the divergence form comes from a material law and many mathematical
models of physical phenomena (if not all) are in divergence form.

The space-time formulation, (written in terms of a flux $J_0$, and
with $y=(x,t)$) then is:
\begin{equation}\label{space-time-div}
\begin{array}{rcl}
&&-\operatorname{div}_y J_0(u) = f, \quad \widetilde J_0(u) = D_0(x)\nabla_x u -
   \bb  u\\
&& u(y)=0,\quad x\in \Gamma=\partial\Omega\times\{(0,t_{\max{}}]\},\\
&& u(y)=u_0(x),\quad x\in \Gamma_0=\overline{\Omega}\times\{t=0\}. 
\end{array}
\end{equation}
Here we have introduced the semidefinite, tensor valued function
$D_0(x)$, and more generally, we denote,
$D_{\varepsilon}(x):\Omega\mapsto \mathbb{R}^{(d+1)\times(d+1)}$:
\begin{equation} \label{Deps}
D_{\varepsilon}=\begin{pmatrix}
K(x)&0\\
0&\varepsilon 
\end{pmatrix},  \quad J_{\varepsilon} = D_{\varepsilon}\nabla_y u -
\bb u.
\end{equation}
The well known heat equation, $u_t-\Delta u = f$, corresponds to
$K(x)=I$ and $\bm{\beta} =0$ and $\varepsilon=0$.

The technique described in the previous section does
not work in a straightforward fashion in the case of space-time
formulation, because $D_0$ is a singular matrix. In fact, there is no
obvious construction that works in the  case of singular
$D_0$. We consider then a formulation using perturbation of the diffusion
tensor $D_\varepsilon$ and the flux $J_{\varepsilon}$.
Thus, for the parabolic problem we set
\begin{eqnarray*}
&& J_{\varepsilon}(u)= D_{\varepsilon}\nabla_y u - \bb u,\qquad  \bb   =(\bm{\beta}^T,1)^T,
\end{eqnarray*}


\subsection{Lowest order discretization for parabolic equations}
In this section we discuss the Scharfetter-Gummel discretization 
  when applied to space-time formulation of a parabolic equation, 
in the lowest order case. As a simple, but
  important example, we consider the simple case of heat equation,
  i.e. $\bm{\beta}_K=0$, which implies that $\bb=\be_{d+1}$,
  and $\be_{d+1} = (\underbrace{0,\ldots,0}_d,1)^T$.  

  We next compute the action of the local stiffness matrix
  corresponding to a parabolic problem on a vector of degrees of
  freedom $u$ representing a function in $V_h$. We fix an element
  ($(d+1)$ dimensional simplex) $T\in \mathcal{T}_h$ and we denote its
  barycentric coordinates by $\{\lambda_i\}_{i=1}^{d+2}$ and the
  space-time coordinates of its vertices are
  $\{\by_i\}_{i=1}^{d+2} = \{(\bx_i,t_i)\}_{i=1}^{d+2}$.  The degrees
  of freedom of a linear polynomial $u\in V_h$ restricted to $T$ are
  $\{u_i\}_{i=1}^{d+2}= \{u(\by_i)\}_{i=1}^{d+2}$ and we have
  $u(\by) = \sum_{i=1}^{d+2} u_i \lambda_i(\by)$. For an edge
  $E\in T$, $E=(\by_i,\by_j)$, $i=1,\ldots (d+2)$, $j=1,\ldots (d+2)$,
  we denote
\[
\btau_{ij}=\btau_E=
\frac{(\by_i-\by_j)}{|\by_i-\by_j|},\quad
|\br|=\sqrt{\sum_{l=1}^{d+1}r_l^2}, \quad \mbox{for all}\; \br\in \mathbb{R}^{d+1}.
\] 
We note that $\btau_{ij}=-\btau_{ji}$, but as we shall see, this is of
no consequence for the final form of the local stiffness matrix.  
To avoid complications in the presentation coming from unnecessary subscripts 
we will write $D$ (resp. $J$) instead of $D_{\varepsilon}$ and
(resp. $J_{\varepsilon}$).

For any $u\in V_h$, as $D$ and $J(u)$ are constants on
$T$, we have the following obvious identities from the definition of
$J$:
\begin{eqnarray*}
&& D^{-1} J \cdot \btau_E = e^{t/\varepsilon}\nabla_{y} \left(e^{-t/\varepsilon} u\right),\\
&& \int_E e^{-t/\varepsilon} D^{-1} J \cdot \tau_E dE= 
\int_E \nabla_{y} \left(e^{-t/\varepsilon} u\right)\cdot \tau_E dE,\\
&&( D^{-1} J \cdot \tau_E)\int_E e^{-t/\varepsilon} dE=
[e^{-t_i/\varepsilon}u(\bx_i,t_i) - e^{-t_j/\varepsilon}u(\bx_j,t_j)]. 
\end{eqnarray*}
Computing the integral on the left side gives
\begin{eqnarray*}
\int_E e^{-\frac{t}{\varepsilon}}dE & = & 
|E|\int_0^1 \exp\left(-\frac{t_j+s(t_i-t_j)}{\varepsilon}\right)ds \\
& = & 
\frac{|E|\varepsilon}{t_j-t_i}\int_{-t_j/\varepsilon}^{-t_i/\varepsilon} e^{\xi}d\xi  = 
|E|\varepsilon\frac{e^{-t_i/\varepsilon}-e^{-t_j/\varepsilon}}{t_j-t_i}\\
& = & 
\frac{|E|e^{-t_i/\varepsilon}}{B(\frac{t_i-t_j}{\varepsilon})}=
\frac{|E|e^{-t_j/\varepsilon}}{B(\frac{t_j-t_i}{\varepsilon})},
\end{eqnarray*}
where $B(s) = \frac{s}{e^s-1}$ is the Bernoulli function ($B(0)=1$). 
Note that, $B(s) = e^{-s} B(-s)$ and $B(s)>0$ for all $s\in \mathbb{R}$.
We then conclude that on every edge $E$ in $T$ we have:
\begin{equation}\label{d1j}
|E|( D^{-1} J \cdot \tau_E)=
B\left(\frac{t_i-t_j}{\varepsilon}\right)u(\by_i) - 
B\left(\frac{t_j-t_i}{\varepsilon}\right)u(\by_j).
\end{equation}
In the derivation for general order of polynomials we needed the
\Nedelec basis and spaces. In the lowest order case, we can take a
route that does not use these spaces explicitly.  In the evaluation of
the stiffness matrix entries, we need to compute integrals of the form
\[
\int_T(J\cdot \nabla_y\lambda_j)=|T|(J\cdot \nabla_y\lambda_j).
\] 
We note that since $D^{-1}J$ is a constant on $T$, we can write it as a gradient of a linear function, namely
\begin{equation}\label{d1jgrad}
J = D(D^{-1}J)=D\nabla_y (D^{-1}J\cdot \by) = 
\sum_{i=1}^{d+2}(D^{-1}J\cdot\by_i)D\nabla_y\lambda_i.
\end{equation} 
Since $\sum_{i=1}^{d+2} \nabla_y \lambda_i\equiv 0$ on $T$, we have that 
\begin{equation*}
0=(D^{-1}J\cdot\by_j)(D\sum_{i=1}^{d+2}\nabla_y\lambda_i\cdot\nabla_y\lambda_j)=
\sum_{i=1}^{d+2}(D^{-1}J\cdot\by_j)(D\nabla_y\lambda_i\cdot\nabla_y\lambda_j)
\end{equation*} 
Hence, 
\begin{eqnarray*}
|T|(J\cdot \nabla_y\lambda_j) &=&  
\sum_{i=1}^{d+2}(D^{-1}J\cdot \by_i)(D\nabla_y\lambda_i\cdot\nabla_y\lambda_j)\\
&=&\sum_{i\neq j}(D^{-1}J\cdot(\by_i-\by_j))(D\nabla_y\lambda_i\cdot\nabla_y\lambda_j)\\
&=&\sum_{i\neq j}|E|(D^{-1}J\cdot\btau_{ij})(D\nabla_y\lambda_i\cdot\nabla_y\lambda_j). 
\end{eqnarray*} 
We have computed earlier (see~\eqref{d1j}) the quantity 
$|E|(D^{-1}J\cdot\btau_{ij})$ for all $E\subset\partial T$. Therefore,  
\begin{equation}\label{stiff}
|T|(J\cdot \nabla_y\lambda_j) =
\sum^{d+2}_{i=1; i\neq j} d^T_{ji}\left\lbrack  
B\left(\frac{t_i-t_j}{\varepsilon}\right)u_i - 
B\left(\frac{t_j-t_i}{\varepsilon}\right)u_j 
\right\rbrack.
\end{equation} 
Here $d^T_{ji}=\int_TD\nabla_y \lambda_i\cdot\nabla_y \lambda_j$ 
are the entries of the local stiffness matrix corresponding to the 
discretization of $(-\operatorname{div} D\nabla)$ with linear elements on $T$. 
Therefore on $T$ we get
\begin{equation}\label{matrix}
[A_T]_{jj} =  -\sum_{i=1;i\neq j}^{d+2}d^T_{ji}B\left(\frac{t_j-t_i}{\varepsilon}\right), 
\quad 
[A_T]_{ji} =  d^T_{ji}B\left(\frac{t_i-t_j}{\varepsilon}\right). 
\end{equation}
The global stiffness matrix is assembled from $A_T$. It is invertible
for sufficiently small mesh size, invertible whenever
the assembly of $d_{ij}^T$ gives an $M$-matrix. 

For more detailed discussions about sufficient conditions which lead
to a stiffness matrix which is an $M$-matrix, as well as relations to
finite volume methods we refer to
\cite{1998BankR_CoughranW_CowsarL-aa}.  More importantly, the
work~\cite{1998BankR_CoughranW_CowsarL-aa} provides techniques for
consistent modification of the local stiffness matrices, leading to
solvable linear systems for wide range of meshes.  In 2 dimensions, a
sufficient condition for the stiffness matrix to be an $M$-matrix is
that the triangulation is a Delaunay triangulation which is easily
achieved by any standard mesh generator. For spatial dimensions
greater than 2, meshes satisfying the condition given
in~\cite[Lemma~2.1]{1999XuJ_ZikatanovL-aa} yields discretization with
$M$-stiffness matrix. If this condition is violated by the mesh, then
the techniques proposed in~\cite{1998BankR_CoughranW_CowsarL-aa} can
be used to modify the consistently the local stiffness matrices so
that the resulting global stiffness matrix is an $M$-matrix.

\section{Numerical tests}\label{section:numerical-tests}
We consider the 2D heat equation with Dirichlet boundary conditions on
the unit square $(0,1)\times (0,1)$. We test both schemes:
StreamLineDIffusion (SLDI) and EAFE on a uniform
triangulation of the unit square. The exact solution is
$$
U(x,t) = e^{-t} \sin\pi x\sin \pi y. 
$$
The domain is the unit square in 2D, and the space-time
problem is solved as fully coupled 3D convection diffusion
problem. 
\begin{figure}[!htb] \centering
\subfloat[]
{
\includegraphics*[width=0.5\textwidth]{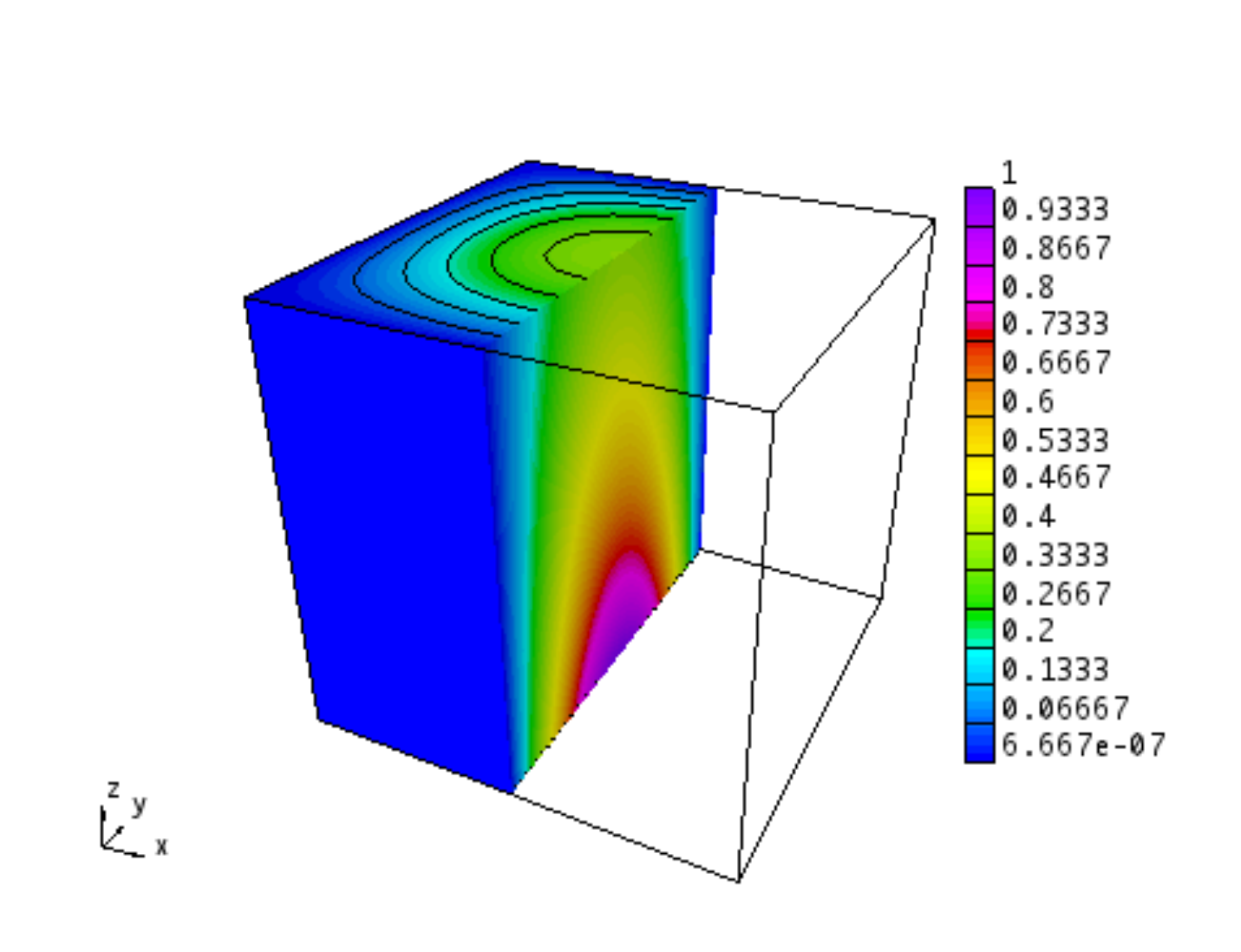}
\label{solution}
}\hspace{3em}
\subfloat[]
{
\includegraphics*[height=0.4\textwidth]{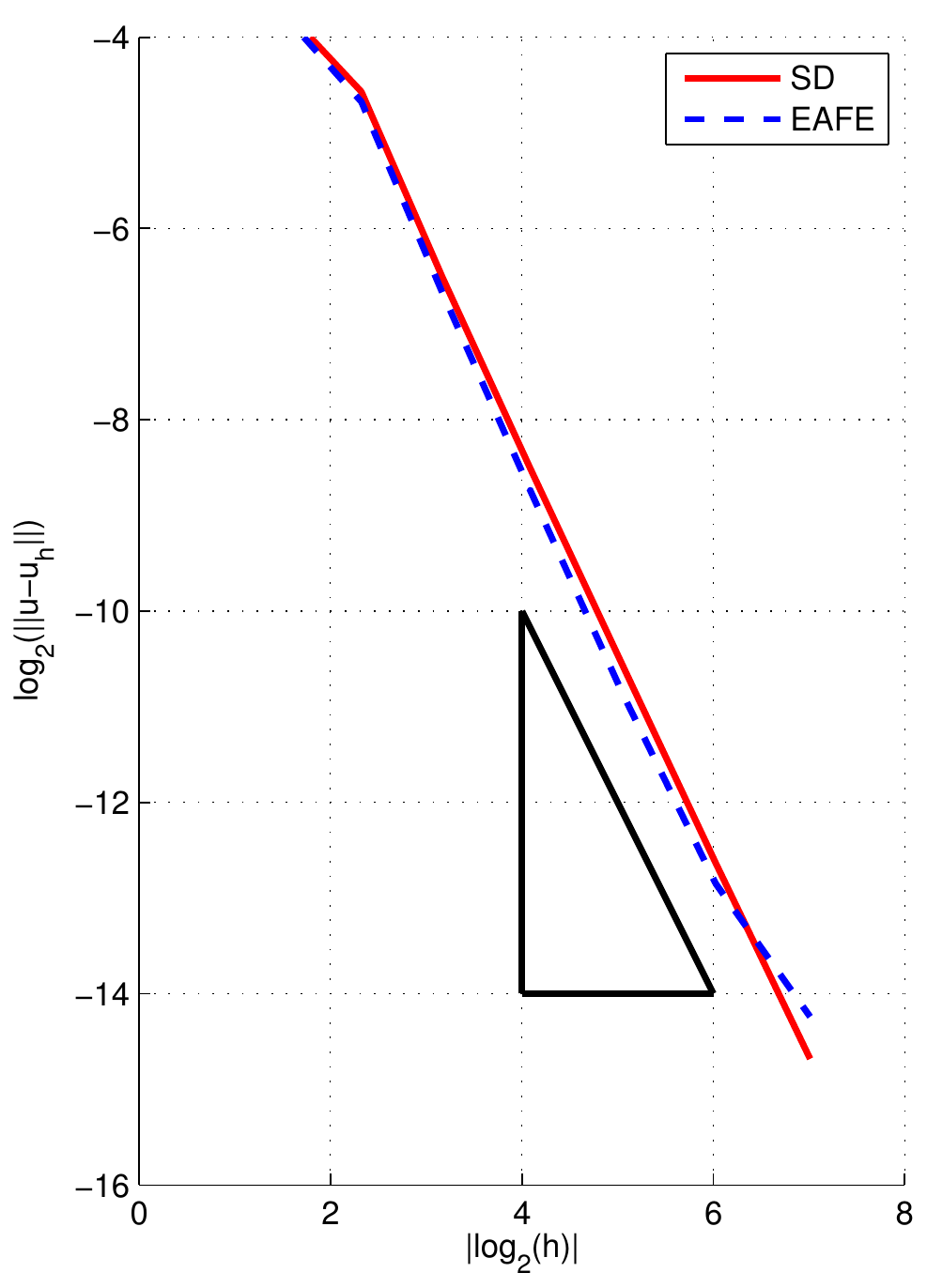}
\label{convergence}
}
\caption{(a) 
Trace of the solution on the plane $x=\frac12$. (b) Error reduction in
$L^2$-norm. Quadratic convergence is  clearly observed.}
\label{convergence-and-solution}
\end{figure}

We have tested the lowest order streamline diffusion
scheme which has the same number of degrees of freedom as the EAFE scheme. 
The convergence behavior of both discretizations is shown in~Figure~\ref{convergence}. 

We have tested the
convergence on a family of successively refined triangulations. The
coarsest one has a mesh size $h_0\approx\frac12$ and the finest
$2^{-8}$ in 3D. The parameter $\varepsilon$ in the diffusion tensor
$D_\varepsilon$ for the EAFE scheme was $10^{-5}$ on all grids. 
The parameter $\theta$ in the streamline
diffusion method was set to $10^{-2}$ on all grids.  Such
pool of tests corresponds to mesh with 27 vertices on the coarsest
grid, and, $\approx 2.1\times 10^6$ vertices on the finest grid.
In~Figure~\ref{solution} we have plotted the trace of the approximate
solution on the plane $x=0$. The approximate solution obtained via the
EAFE scheme looks exactly the same, as is also the exact solution. 

We next show a plot of a solution to an equation with 
convection depending on time. The equation is
\begin{equation}\label{reb}
 u_t - \operatorname{div}(K(x)\nabla u - \bb u)= 1,\quad x \in \Omega_s,
\quad \bb = \begin{pmatrix}
100\sin(6\pi t)\\
0
\end{pmatrix}
\end{equation}
and the boundary and initial conditions are homogeneous, the domain is the unit square and the time interval is $(0,1)$. The solution
via the Scharfetter-Gummel (EAFE) scheme is shown in Figure~\ref{fig-2}. Note
that with such convection term, the convection is $0$ for $t=k/6$ and
$k$ integer; it is, however, convection dominated for other values of $t$. 
\begin{figure}[!htb] \centering
\includegraphics*[width=0.625\textwidth]{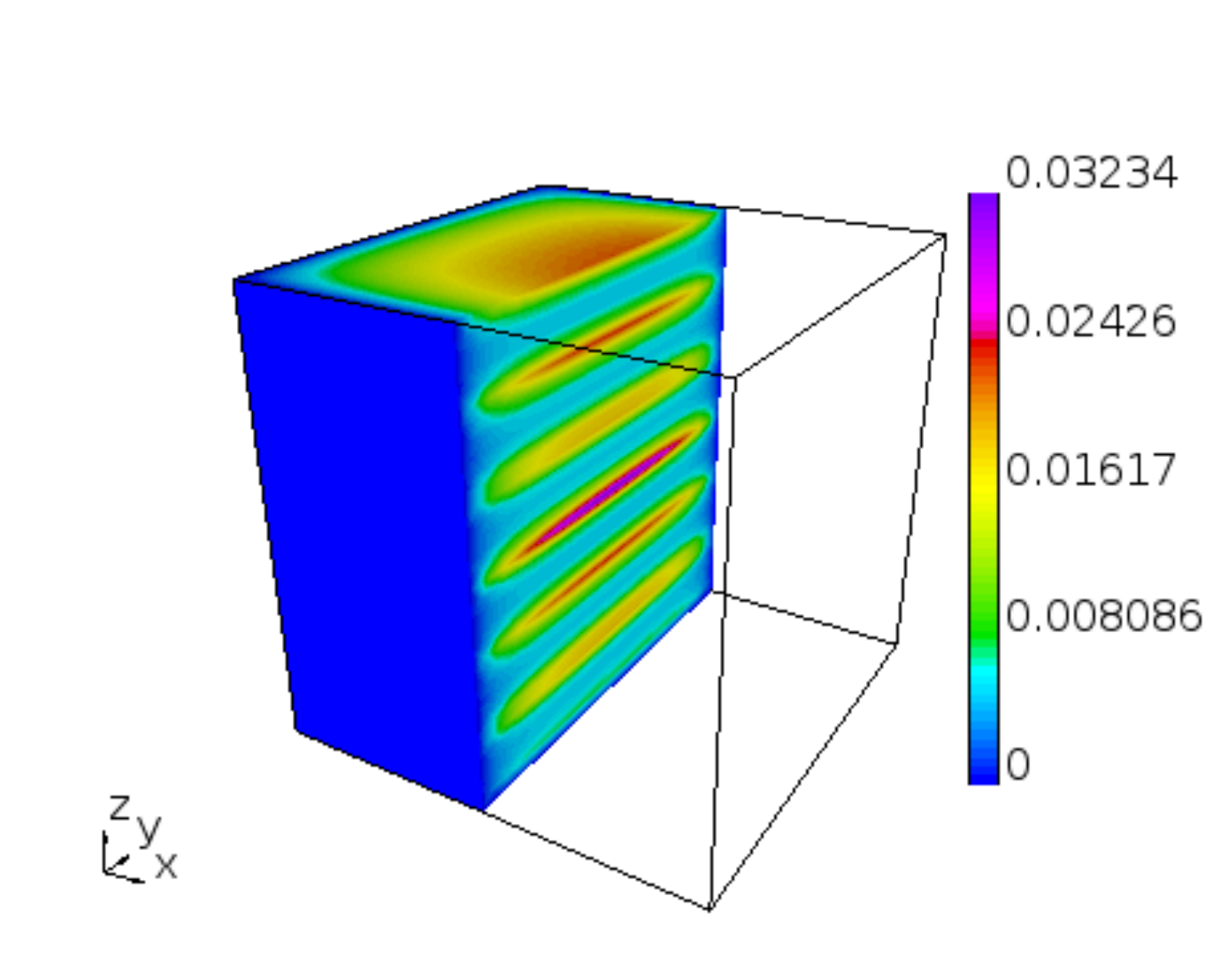}
\caption{Trace of the numerical solution of equation~\eqref{reb} on
  the plane $x=\frac12$. The effect of the time dependent convection
  is clearly seen in the plot. \label{fig-2}}
\end{figure}

We have mentioned already the software used in performing the
tests. In summary, we have used the \verb|C++| library and examples
from the \verb|mfem| package~\cite{mfem-library} (discretization); The solutions of the
resulting linear systems are done using the Algebraic Multigraph
Multilevel ILU algorithm by Bank and
Smith~\cite{2002BankR_SmithR-aa,multigraph-solver} found
at~\url{http://ccom.ucsd.edu/~reb/software.html}. The visualization
was done using the \verb|glvis| tool~\cite{glvis-tool}. 

\section{Concluding remarks}
We introduced a class of numerical methods for convection diffusion
equations in arbitrary spatial dimensions. In principle, these schemes
can be applied to wide range of problems, such as linearization of the
Nernst-Plank equations for transport of species in a charged media and
the space-time discretizations of such equations. We have derived
novel exponentially fitted (higher order Scharfetter-Gummel and
streamline diffusion) discretizations for convection diffusion
equations.  Distinctive features of the proposed Scharfetter-Gummel
discretization are: (1) its the monotonicity in the lowest order case;
(2) its applicability in any spatial dimension; and (3) the a priori
estimates are in terms of the flux only.  For order higher than 1, the
derivation and the analysis of this scheme is new, and its
implementation for space-time formulation of parabolic problems is a
subject of a current research.

\appendix
\section{Remarks on unisolvence for the quasi-polynomial \Nedelec space}\label{appendix}
To keep the presentation self contained, in this appendix, we give a
proof of Lemma~\ref{lemma-unisolvence} using vector proxies.  Recall
that Lemma~\ref{lemma-unisolvence} states that the \Nedelec degrees
of freedom are unisolvent set of functionals for the quasi-polynomial
space
\[
(eP_{r-1})^d=\{e^{(-\bq\cdot \bx)}D^{-1}\bp \;\big|\; \bp\in (P_{r-1})^d\}.  
\] 
To prove Lemma~\ref{lemma-unisolvence} using vector proxies in
$\mathbb{R}^n$ and with arbitrary polynomial degree $r$ we recast some
of the results from~\cite[Section~4]{2006ArnoldD_FalkR_WintherR-aa}
for the quasi-polynomial spaces in terms of vector proxies.  This is
done with the aim to keep the notation in accordance with what we
used before.

\subsection*{Quasi-polynomial spaces: A two dimensional example} To
illustrate the difficulties that occur when using quasi-polynomial
spaces instead of polynomial spaces, let us look closely at the case
$\Omega\subset\mathbb{R}^2$, i.e. $\mathcal{T}_h$ is a
triangulation. The \Nedelec degrees of freedom are the $0$-th moments
on the edges of the triangulation, namely
\(\langle \eta_e,v\rangle = \int_e v\cdot\tau_e\).
The polynomial space $P^{\mathcal{N}}$ on a
$T\in \mathcal{T}_h$ is the space of vector valued linear polynomials
of the form:
\[\bp(\bx) = \ba + \beta \begin{pmatrix}-x_2\\ x_1\end{pmatrix}.\]
Clearly, we have that $(P_0)^d\subset P^{\mathcal{N}}\subset (P_1)^d$. 

Next, to show that the set of \Nedelec degrees of freedom is a unisolvent
set of degrees of freedom for $P^{\mathcal{N}}$ it is sufficient to show that:
If $\bp\in P^{\mathcal{N}}$ and $\langle\eta_e,\bp\rangle=0$ for
all $e\in \partial{T}$, then $\bp\equiv 0$. The proof follows by
integration by parts:
\[
0=\sum_{e\in \partial T}\int_{e} \bp\cdot \bm{\tau}_e = -\int_T \operatorname{rot} \bp=|T|\beta. 
\]
This implies that $\beta=0$, and hence $\bp=\ba$ is a constant
vector. As the degrees of freedom vanish we have that
\[
0=\int_e\bp\cdot \bm{\tau}_e = |e|(\bp\cdot \bm{\tau}_e), 
\quad \mbox{for all}\quad e\in \partial T. 
\]
This implies that $\bp\equiv 0$, because any two of the edges of $T$
are linearly independent for a non-degenerate simplex
$T\subset \mathbb{R}^2$ and $\bp$ is a constant vector orthogonal to
them.

Let us now look at the case of quasi-polynomial space. We set for
simplicity $D=I$ and consider the quasi-polynomial space
\[
eP^{\mathcal{N}}=\{e^{(-\bq\cdot \bx)}\bp \;\big|\; \bp\in P^{\mathcal{N}}\}. 
\] 
It is easy to see that first part of the argument given above
fails. Indeed, if $\bp\in eP^{\mathcal{N}}$ we have
\begin{equation}\label{z}
0=\sum_{e\in \partial T}\int_{e} e^{-\bq\cdot\bx}\bp\cdot \bm{\tau}_e 
= -\int_T \operatorname{rot} (e^{-\bq\cdot\bx}\bp).
\end{equation}
In general, the above relation does not imply that $\beta=0$ and we
cannot conclude from~\eqref{z} that $\bp$ is a constant.

If we reduce the degree of the polynomial space, however, the argument
works. Indeed, let us take $\bs\in (eP_0)^2$ such that
$\bs=e^{-\bq\cdot\bx} \bp$ for some $\bp\in (P_0)^2$ and assume
$\langle\eta_e,\bs\rangle=0$ for all $e\in \partial T$. We then have
\[
0=
\int_e e^{-\bq\cdot\bx}\bp\cdot \bm{\tau}_e = 
(\bp\cdot \bm{\tau}_e)\int_e e^{-\bq\cdot\bx}, 
\]
and this clearly gives $(\bp\cdot \bm{\tau}_e)$ for all
$e\in \partial T$, which tells us that $\bp\equiv 0$ and hence $\bs\equiv 0$. 

As a conclusion we have shown that the \Nedelec degrees of freedom are
unisolvent set of degrees of freedom (functionals) for $(eP_0)^2$, but not
necessarily for $eP^{\mathcal{N}}$. Generalizing such statement to
arbitrary spatial dimension and polynomial degrees requires a bit more
work and is detailed in what follows.

\subsection*{Simplices and barycentric coordinates}
As $D$ is a constant matrix, we may assume $D=I$, because
this does not change the degree of the quasi-polynomial space. 

The following itemized list contains well known definitions and facts
related to the geometry of simplices in $\mathbb{R}^d$. 
\begin{itemize}
\item The vertices of $T$ are denoted by $[\bx_0,\ldots,\bx_d]$. We
  assume that we have ordered the vertices in such a way that
  $\{\bm{\tau}_m=(\bx_m-\bx_0)\}_{m=1}^{d}$ form a positively oriented basis in
  $\mathbb{R}^d$. 

\item For $k$-vectors $\{\bv_j\}_{j=1}^k$ in $\mathbb{R}^n$, we
  introduce the volume of the $k$-simplex with one vertex at the
  origin and the other vertices at the coordinates of the
  $k$-points. We set
\[
\operatorname{vol}(\bv_1,\ldots,\bv_k) = \pm \frac{1}{k!}\sqrt{\det(V^T V)}, \quad
V=(\bv_1,\ldots,\bv_k)\in \mathbb{R}^{n\times k}. 
\]
The plus sign is chosen if $\{\bv_1,\ldots,\bv_k\}$ is a positively
oriented frame.  The volume is obviously zero if the vectors are linearly
dependent. 

\item The set of all sub-simplices of $T$ of dimension $j$ can be
  identified by the monotone set of indices $\sigma=[i_0,\ldots i_j]$, with
  $i_0< \ldots <i_j$. For any such set of indices, we denote by
  $\sigma^{c}$ its complement in $\{0,\ldots, d\}$. 
The elements of $\sigma$ and $\sigma^c$ are always assumed to be in increasing
  order. In general $s_\sigma$ will denote a sub-simplex of $T$ of
  dimension $|\sigma|$ (with $s_{\{1:d\}}=T$), and every such simplex is the
  convex hull of any $(|\sigma|+1)$ vertices of $T$, i.e.
  $s_\sigma=\operatorname{chull}(\bx_{\sigma})=\operatorname{chull}(\bx_{i_1},\ldots \bx_{i_j})$, with  $i_0< \ldots < i_j$.


\item There is a canonical affine map $\Phi: \widehat T \mapsto T$
  with domain the simplex $\widehat T\subset \mathbb{R}^d$ with vertices
  $[\bm{0},\be_1,\ldots,\be_{d}]$. This map is described by a
  $(d\times d)$ matrix $B$, and
\[
\Phi(\widehat \bx) = \bx_0 + B \widehat \bx, \quad
B=(\btau_1 | \ldots | \bm{\tau}_d). 
\]
If $\widehat s_j$ is the reference simplex
in $\mathbb{R}^j$ with vertices  $[\bm{0},\be_1,\ldots,\be_{j}]$, then we also have a map (induced by
$\Phi$), $\Phi_j:\widehat s_j\mapsto s_\sigma$ 
\[
\Phi_\sigma(\widehat \bx) = \bx_0 + B_\sigma\widehat \bx, \quad
B_{\sigma}=(\bm{\tau}_\sigma)=(\btau_{l_1}| \ldots | \bm{\tau}_{l_j}), \quad \widehat \bx\in
\widehat s_j. 
\]

\item $\{\lambda_j\}_{j=0}^d$ denote the barycentric coordinates of
  $T$. Note that $\{\nabla \lambda_m\}_{m=1}^d$ form a basis for
  $\mathbb{R}^d$. In fact, with the definitions above it is easily
  verified that
\[
  (\nabla\lambda_m,\bm{\tau}_j)_{\ell^2}=
  (B^{-T}\nabla_{\widehat x}\widehat \lambda_m,B\be_j)_{\ell^2}=
  (\nabla_{\widehat x}\widehat \lambda_m,\be_j)_{\ell^2}=\delta_{mj}. 
\]
In addition, for any vector field $\bu(\bx)\in (P_{r-1})^n$ and any
$n$, we can use the basis $\{\nabla \lambda_j\}_{j=1}^n$ to write 
\begin{equation}\label{lambda-expansion}
\bu(\bx) = \sum_{j=1}^n u_j \nabla\lambda_j, \quad u_j(\bx)\in
P_{r-1}. 
\end{equation}
\item Let $s_{\sigma}$, $\sigma=(l_1,\ldots,l_k)$ be a subsimplex of
  $T$.  Let $\nabla_s\lambda_j$, $j\in\sigma$ be the surface gradient
  of $\lambda_j$ when on the hyperplane spanned containing
  $s_{\sigma}$. This is another basis for
  $\operatorname{span}\{\bm{\tau}_{\sigma}\}$ as is easily verified.
  We note that $\nabla_\sigma\lambda_j= \pi_\sigma\nabla\lambda_j$, where
  $\pi_\sigma$ is an orthogonal projection on the span of $(\bm{\tau})_{\sigma}$. We have
\[
\pi_\sigma \nabla\lambda_j = B_{\sigma} (B_\sigma^TB_{\sigma})^{-1} B_{\sigma}^T\nabla\lambda_j 
= B_{\sigma} (B_\sigma^TB_{\sigma})^{-1}\nabla_{\widehat\bx}\widehat \lambda_j. 
\] 
Note that $\pi_{\sigma}\nabla \lambda_k=0$ if $k \notin \sigma$ (the 
gradient is normal to the zero level set of $\lambda_k$), and, we also have, 
$\pi_\sigma \nabla\lambda_{l_j} =
B_{\sigma}(B_\sigma^TB_{\sigma})^{-1}\be_{j}$ for $j=1,\ldots,k$,
where $\be_j\in \mathbb{R}^k$.

\item By trace of a vector field $\bv$ on a sub-simplex $\sigma$ we mean 
\[
\operatorname{Tr}_{\sigma}(\bv) = \pi_\sigma \bv(\bx), \quad \bx\in s_\sigma. 
\]
\item For an $n$-dimensional simplex $s\subset \mathbb{R}^n$,
  $s=[\bx_1,\ldots\bx_n]$ $n\le d$ the zero level sets of $\lambda_j$
  we denote by $\{F_j(s)\}_{j=0}^{n}$ are often also referred to as
  ``faces'' of $s$, and are the $(n-1)$ dimensional simplices opposite
  to the $0$-dimensional simplices (vertices) of $s$. In short,
  $F_{j}(s)=s_{\sigma}$ faces the vertex $\bx_j$. Here, for a fixed
  $j$, $\sigma=\{0,\ldots,j-1,j+1,\ldots,n\}=\{j\}^c$. 

\end{itemize}

Next, using the facts above, we now prove a well known
result~\cite[Lemma~4.5]{2006ArnoldD_FalkR_WintherR-aa} and use the
language of vector proxies.The changes in the proof
of~\cite[Lemma~4.5]{2006ArnoldD_FalkR_WintherR-aa} needed to account
for the fact that we have a quasi-polynomial space, instead of
polynomial space, are obvious. Note that the face $F_0$ is excluded
from the statement of the lemma.
\begin{lemma}\label{bubbles}
Let $s$ be a simplex in $\mathbb{R}^n$ and 
$\bu\in (eP_{r-1})^n$ be such that its traces on the faces
of $s$, $\{F_j(s)\}_{j=1}^n$ vanish. Then, 
\[
\bu = \sum_{j=1}^n\left(\prod_{k\neq j} \lambda_k\right)\widetilde{u}_j \nabla \lambda_j, 
\] 
where $\widetilde{u}_j\in eP_{r-n}$. 
\end{lemma}
\begin{proof}
We use~\eqref{lambda-expansion} to write 
\begin{equation}
\bu(\bx) = \exp(-\bq\cdot \bx)\left(\sum_{j=1}^n u_j \nabla\lambda_j\right),
\end{equation}
with $u_j\in P_{r-1}$. Let us fix a $k$ and recall that the trace of
$\bu$ on $F_k(s) =  \{\bx \;\big|\;\lambda_k=0\}$ vanishes. As we
pointed out, 
$\pi_{F_k} \nabla\lambda_j$, $j\neq k$, form a basis on the hyperplane
containing $F_k$. Further, for $\bx\in F_k$, by the
definition fo the trace of $\bu$ on $F_k$, we have
\[
\operatorname{Tr}_{F_k}\bu(\bx) = 
\exp(-\bq\cdot \bx)\sum_{j=1}^n u_j \pi_{F_k}\nabla\lambda_j
=\exp(-\bq\cdot \bx)\sum_{j\neq k} u_j \pi_{F_k}\pi_{F_k}\nabla\lambda_j,
\]
As $\{\pi_{F_k}\nabla\lambda_j\}_{j\neq k}$ form a basis for the
hyperplane containing $F_k$, we have that $v_j=0$ for $\bx\in F_k$ for
all $j\neq k$ and hence $u_j = \lambda_kv_j$, where $v_j\in
P_{r-2}$. Applying this argument for $k=1,\ldots,n$ concludes the
proof. 
\end{proof}

\subsection*{Basic facts about differential forms} 
When next introduce some basic facts about differential forms. We need
these to define the action of the functionals $\eta_j$, which we
introduced in~\S\ref{section:Nedelec}. In $\mathbb{R}^3$ these were
introduced by \Nedelec~\cite{1980NedelecJ-aa,1986NedelecJ-aa} (see
also~\cite{1988BossavitA-aa,2003MonkP-aa}). Our goal here is to prove
unisolvence for any spatial dimension and any polynomial degree. That
brings us into the framework of Finite Element Exterior Calculus, in
the spirit
of~\cite{1999HiptmairR-aa,2006ArnoldD_FalkR_WintherR-aa,2010ArnoldD_FalkR_WintherR-aa}.
Of course, we restrict our discussion to vector valued functions (a
flux $J(u)$, and,
hence we can drop some of the generalities since we consider
differential $1$-forms only.

The space of differential $k$-forms on $T$ we denote by $F\Lambda^k$
\begin{eqnarray*}
  &&F\Lambda^{k}=\{\omega\;\big|\;\omega = 
\sum_{\sigma}f_\sigma(\bx)dx_\sigma\}, \;\; f_\sigma(\bx)\in C^{\infty}(T). 
  \\ 
  && \sigma=\{i_1,\ldots i_l\}\subset\{1,\ldots d\},
     d\bx_\sigma = dx_{i_1}\wedge\ldots\wedge dx_{i_l}, 
\end{eqnarray*}
The summation is over all subsets $\sigma\subset \{1,\ldots,d\}$ of
length $|\sigma|=k$. 

If the coefficients $f_\sigma$ are polynomials of degree at most
$(r-1)$, \textit{i.e.} $f_\sigma\in P_{r-1}$ (resp.
$f_I\in eP_{r-1}$), we denote the polynomial space of such $k$-forms
by $P_{r-1}\Lambda^k$ (resp.  $eP_k\Lambda^k$). By convention, all
these spaces are trivial if $k<0$ or $r<1$. The exterior product is
denoted by ``$\wedge$'' and the symbol
``$d\omega\in \mathcal{D}F^{k+1}$'' denotes the exterior derivative of
a form $\omega\in F\Lambda^k$. By definition we have
\[
\mbox{If}\;\;\omega = w(\bx) d x_\sigma,\; \;\mbox{then}\;\;
d\omega = \sum_{j=1}^d \partial_j w(\bx) dx_j\wedge dx_{\sigma}.
\]
A basis in the space of $k$-forms is given by 
$d x_\sigma=dx_{l_1}\wedge\ldots\wedge dx_{l_k}$. 
On a simplex $T$ such
basis is also provided by $d\lambda_j$, where $\lambda_j$,
$j=1,\ldots,d$ are the barycentric coordinates in $T$ as described
earlier. 

Let us note, though, that we do not need the exterior derivative to
prove the unisolvence result. We introduced it to clarify some
differences between using quasi-polynomial spaces, instead of regular
polynomials.  The functionals $\eta_j$ are dual to the \Nedelec basis,
which spans the classical \Nedelec space of degree $r$:
$P^{\mathcal{N}}_{r} = P_{r-1}\Lambda^1\oplus \widetilde
P_{r}\Lambda^1$,
where $\widetilde P_{r}$ denotes the space of all homogenous
polynomials of degree $r$. Clearly, from this we see that
$\dim P^{\mathcal{N}}_{r} > \dim P_{r-1}\Lambda^1$.

Next, when proving unisolvence for regular polynomial spaces, 
we have that the exterior differentiation decreases the degree of the
polynomials by 1, namely, 
$dP_{r-1}\Lambda^k\subset P_{r-2}\Lambda^{k+1}$. For the
quasi-polynomial space this is not the case, and we only have
$dP_{r-1}\Lambda^k\subset P_{r-1}\Lambda^{k+1}$. As a consequence, the
proofs of unisolvence of $\eta_j$ on the space of polynomials do not
extend directly to the case of quasi-polynomials generated by
$P^{\mathcal{N}}_r$. This explains our choice in
\S\ref{section:Nedelec} of $eP_{r-1}\Lambda^1$ instead of
$eP^{\mathcal{N}}_r$. 
 

\subsection*{Integration of forms}  
A $k$-form can be integrated on a subsimplex of dimension $k$ (only),
and, for any such subsimplex $s_{\sigma}\subset T$, and any
$\omega=w(\bx)dx_L$, where
$L=\{l_1,\ldots,l_k\}\subset \{1,2,\ldots,d\}$,
$dx_L=dx_{l_1}\wedge\ldots\wedge dx_{l_k}$ we define (see, for
example, Rudin~\cite{1976RudinW-aa}):
\[
\int_{s_{\sigma}}\omega := \int_{\widehat{s}_k} w(\Phi_{\sigma}(\widehat{\bx}))
\frac{\partial_L\Phi_\sigma}{\partial(\widehat x_1,\ldots,\widehat
  x_l)}d\widehat x_1\ldots d\widehat x_k.
\]
This is just a definition of an integral of a differential form, and,
the integral on the right side is to be understood as ordinary
integral over the reference $k$-simplex
$\widehat s_k\subset \mathbb{R}^k$.  The map
$\Phi_{\sigma}:\widehat s_{\sigma}\mapsto \mathbb{R}^d$ was defined
earlier and
$\frac{\partial_L\Phi_\sigma}{\partial(\widehat x_1,\ldots,\widehat
  x_k)}$
denotes the determinant of a $k\times l$ submatrix of its Jacobi
matrix corresponding to the indicies in $L$:
\[
\frac{\partial_L\Phi_\sigma}{\partial(\widehat x_1,\ldots,\widehat
  x_l)}=
\det
\begin{pmatrix}
\partial_1 \Phi_{\sigma,i_1} &
\ldots & \partial_l \Phi_{\sigma,i_1} \\
\vdots& \vdots& \vdots\\
\partial_1 \Phi_{\sigma, i_l} &
\ldots & \partial_l \Phi_{\sigma,i_l}
\end{pmatrix}, \quad \partial_j \Phi_{\sigma,m}=
\frac{\partial\Phi_{\sigma,m}}{\partial \widehat x_j}.
\]
In our case $\Phi_\sigma$ is an affine map, so the Jacobi matrix and its
determinant are independent of $\widehat \bx$. 
Let $B$ be the matrix that
maps $\be_j\mapsto\bm{\tau}_j$, $j=1:n$. For a monotone 
$L\subset\{1,\ldots,d\}$, $\sigma=\{i_1,\ldots,i_k\}$ 
we define $I_L\in \mathbb{R}^{d\times k}$ to be the matrix whose
range is
$\operatorname{span}\{\be_{l_1},\ldots,\be_{l_k}\}$. The definition is
\[
I_{L} := \sum_{j=1}^k \be_{l_j} [\be_j(\mathbb{R}^k)]^T.
\]
Note that $\be_j(\mathbb{R}^k)$ are the canonical basis vectors in
$\mathbb{R}^k$. Next, if $\omega = \sum_{L} w_L dx_L$ is a differential $k$-form (
(the summation is over all monotone subsets $L$ of $\{1,\ldots,n\}$,
and $L=\{l_1,\ldots,l_k\}$) 
we can write the definiton of the integral above as
\[
\int_{s_{\sigma}}\omega := 
\int_{\widehat{s}_k} \sum_{L}w_L(\Phi_{\sigma}(\widehat{\bx}))\det(B_{\sigma,L})
d\widehat x_1\ldots d\widehat x_k,
\]
where $B_{\sigma,L} = I_\sigma^T B I_L$. 

As we pointed out earlier, we
have $\be_j = B^{-T} \nabla\lambda_j$ and $\be_j = B \bm{\tau}_j$,  and
hence, the vectors tangent to $T$ and $\nabla\lambda_i$ form a
bi-orthogonal basis.  It is often convenient to represent differential
form via the basis $d\lambda_1,\ldots,d\lambda_d$. Thus for a
differential form we will have two representations
\begin{eqnarray*}
\omega &=& \sum_{L} w(\bx) dx_L\\
\omega & = & \sum_{L} v_{L}(\bx) (d\lambda)_L, \quad v_L =
             \sum_{\sigma} \det(B_{L,\sigma}) w_{\sigma}.
\end{eqnarray*}
The identity between the coefficients is easy to verify using that
\[
d\lambda_j  =
\sum_{m=1}^d (\nabla\lambda_j)_m dx_m = 
\sum_{m=1}^d (B^{-1})_{mj} dx_m. 
\]
Differential $k$-forms are elusive objects, they only manifest 
themselves when integrated over $k$-simplices. Otherwise they mean
nothing. The fact
that $\lambda_j$ and the trace of $\nabla\lambda_j$ vanishes on any
sub-simplex where $\lambda_j=0$ shows that
$\int_{s_{\sigma}} v_L(\bx)(d\lambda)_L = 0$ if $L\neq \sigma$. 
This is a nice property, because the change of variable formula 
gives us that
\[
\int_{s_{\sigma}}\omega =  \int_{s_{\sigma}} v_\sigma(\bx) d\operatorname{vol}(s).
\]
where the right side is just the ordinary integral over the simplex
$s_{\sigma}$, 
\[
d\operatorname{vol}(s) = \sqrt{\det(\Lambda^T\Lambda)}d\bx, \quad
\Lambda = 
(\pi_s\nabla\lambda_{i_1},\ldots,\nabla\lambda_{i_k})\in
\mathbb{R}^{d\times k}. 
\]

Exterior product of two differential forms is also used below for
defining degrees of freedom. If $\omega$, $\eta$ are $p$ and $q$ form
respectively, then $(\omega\wedge \eta)$ is a $(p+q)$ form. The
definition and properties of the exterior product ``$\wedge$'' are
found in standard texts.

\subsection*{Degrees of freedom}  
Since we only consider 1-forms, let us focus on this case here and 
the degrees of freedom of the classical \Nedelec space in terms of differential $1$-forms,
and later we convert to definitions involving vector proxies.
Following~\cite[Section~4]{2006ArnoldD_FalkR_WintherR-aa} and 
\cite[Definition~9]{1999HiptmairR-aa}, the degrees of freedom for
\emph{polynomial} differential 1-form $\omega\in P_r^{\mathcal{N}}$ are:
\begin{equation}\label{dof}
\langle \eta_\sigma,\omega\rangle = \int_{s_\sigma} \omega \wedge \widetilde{\eta}_\sigma, 
\quad
\widetilde{\eta}_\sigma\in P_{r-k}\Lambda^{k-1}(s_{\sigma}), 
\end{equation}
Here, the notation is as follows: 
$\sigma$ ranges over all monotone subsets of $\{1,\ldots,d\}$ of
cardinality $k$; $s_{\sigma}$ are the corresponding sub-simplices of
$T$ of dimension $k$; $\{\widetilde{\eta_j}\}$ forms a basis in for $P_{r-k}\Lambda^{k-1}$
and $\eta_j$ is the corresponding functional it generates, as defined
by the integral on the right side. As we pointed out, we consider here
only $1$-forms, which makes things easier. In such case, any
$\eta\in P_{r-n}\Lambda^{n-1}$ can be identified with a proxy-vector,  
$\bm{p}\in (P_{r-n})^n$ by identifying every $(n-1)$ dimensional
$F_i$ simplex with the vertex $i$ it opposes.  
\[
\widetilde{\eta} = \sum_{i=1}^n p_i(\bx)(d\lambda)_{F_i} \mapsto (p_i)_{i=1}^n.
\] 
Integration then of $u\wedge \widetilde{\eta}$ amounts to:
\[
\sum_j\int_{s} u_jp_j 
\operatorname{vol}(\nabla \lambda_j,\nabla\lambda_{1},\ldots\widecheck{\nabla \lambda_j},\ldots,\nabla\lambda_n)d\bx. 
\]
Here the coefficients $u_j$ are the coefficients in the
expansion~\eqref{lambda-expansion}. 

Next we prove a result analogous to the \cite[Lemma~4.6]{2006ArnoldD_FalkR_WintherR-aa}
\begin{lemma}\label{lemma-increase-dimension}
Let $s$ be an $n$-dimensional simplex and  let 
$\bu\in eP_{r-1}\Lambda^1$ be a given vector field with
vanishing traces on
the faces $F_i\subset \partial s$,
$F_i=\{\bx\;\big|\;\lambda_i(\bx)=0\}$, $i=1,\ldots,n$. Assume also
that
\[
\sum_j \int_s  \exp(-\bq\cdot\bx) u_j p_j \operatorname{vol}
(\nabla \lambda_j,(\pi_{F_j}\nabla\lambda_m)_{m\neq j}) = 0, 
\quad\mbox{for all} \bm p\in (P_{r-n}(s))^{n}.
\]
Then, $\bu=0$. 
\end{lemma}
\begin{proof}
Since the traces
of $\bu$
vanish on the faces of $s$, 
from Lemma~\ref{bubbles} we know that
\[
\bu = \sum_{j=1}^n \left(\prod_{m\neq j} \lambda_m\right)\widetilde{u}_j \nabla \lambda_j, 
\] 
where $\widetilde{u}_j\in eP_{r-n}$. Choosing $p_j = \widetilde{u}_j$ then
concludes the proof of the lemma because 
$\prod_{m\neq j} \lambda_m$ is a strictly positive function in the
interior of $s$. 
\end{proof}
\subsection*{Unisolvence result}
After all these long introductory remarks, let us recall that our goal is
to show that $ZP$ defined in \S\ref{derivation-section} is injective. 
This is equivalent to showing that if the functionals defined in~\eqref{dof}  
vanish on a differential 1-form $\omega$
with coefficients in $(eP_{k-1})^d$, then $\omega=0$.  Equivalently we
can state this in terms of vector fields as we did in
Lemma~\ref{lemma-increase-dimension}.
We have the following unisolvence result
\begin{lemma}\label{unisolvence-b}
If $\bu=\sum_{j=1}^d \exp(-\bq\cdot\bx)u_j\nabla\lambda_j$ is a vector field for which
\[
\sum_{j=1}^n\int_{s_{\sigma}} \exp(-\bq\cdot\bx) u_j p_j 
\operatorname{vol}
(\pi_{\sigma}\nabla\lambda_j,(\pi_{\sigma} \nabla\lambda_m)_{m\neq j})
= 0,\quad\mbox{for all}\quad p_j\in (P_{r-n})^n,
\]
where $n=1,\ldots,d$, $|\sigma|=n$, then $\bu=0$. 
\end{lemma}
\begin{proof}
The proof is just an iteration using the result from
Lemma~\ref{lemma-increase-dimension}. Indeed taking $n=0$ and we have
that 
\[
0=\frac{1}{|s|}\int_{s} \exp(-\bq\cdot\bx) \bu\cdot\bm{\tau}_{s} p,
\quad\mbox{for all}\quad p\in P_{r-1},
\]
where $s$ is any of the $1$-dimensional sub-simplices (edges) of
$T$. This clearly tells us that $\pi_s\bu=0$ on every edge. Since two
dimensional simplices have as boundary edges 1-dimensional simplices,
an application of Lemma~\eqref{lemma-increase-dimension} shows that
$\pi_s \bu=0$ also on all two-dimensional  sub-simplices $s$. Formalizing this
as an induction argument is trivial and left to the reader. The proof
is complete. 
\end{proof}



\end{document}